\documentclass[preprint]{imsart}
\arxiv{arXiv:1605.07129}

\usepackage{url}
\usepackage{verbatim}
\usepackage{paralist}
\RequirePackage[OT1]{fontenc}
\RequirePackage[numbers]{natbib}
\RequirePackage{hypernat}
\usepackage{mathscinet}
\usepackage[utf8]{inputenc} 
\usepackage{textcomp} 
\usepackage{graphicx}  
\usepackage{flafter}  
\usepackage{natbib} 
\usepackage{amsmath,amssymb,amsthm}  
\usepackage{etoolbox} 
\usepackage{bm}  
\usepackage[pdftex,bookmarks,colorlinks,breaklinks]{hyperref}  
\usepackage{url}
\usepackage{subfig}
\usepackage{dsfont}

\usepackage{verbatim}

\usepackage{wrapfig}
\usepackage{enumerate}

\hypersetup{linkcolor=blue,citecolor=black,filecolor=cyan,urlcolor=blue} 

\startlocaldefs
\numberwithin{equation}{section}
\theoremstyle{plain}
\newtheorem{definition}{Definition}[section]
\newtheorem{theorem}{Theorem}[section]
\newtheorem*{thm}{Theorem}
\newtheorem{corollary}{Corollary}[section]

\newtheorem{lemma}{Lemma}[section]
\newtheorem{assumption}{Assumption}
\newtheorem{remark}{Remark}

\newtheorem{fact}{Fact}[section]

\allowdisplaybreaks

\endlocaldefs

\newcommand{\dotp}[2]{\left\langle#1,#2\right\rangle}

\newcommand{\m}{\mathcal}
\newcommand{\mb}{\mathbb}
\newcommand\argmin{\mbox{argmin}}
\newcommand{\sign}{\mbox{sign}}

\newcommand{\tr}{\mbox{tr\,}}
\newcommand{\mx}{\mbox{\footnotesize{max}\,}}
\newcommand{\mn}{\mbox{\footnotesize{min}\,}}
\newcommand{\rank}{\mathrm{\,rank}}
\def\r{\right}
\def\l{\left}
\newcommand{\eps}{\varepsilon}
\newcommand{\var}{\mbox{Var}}
\newcommand{\wh}{\widehat}

\begin{document}

\begin{frontmatter}
\title{
Sub-Gaussian estimators of the mean of a random matrix with heavy-tailed entries}
\runtitle{
Estimators of the mean of a random matrix}
\begin{aug}
\author{\fnms{Stanislav} \snm{Minsker}\thanksref{t2}\ead[label=e1]{minsker@usc.edu}}

 \thankstext{t2}{Supported in part by the National Science Foundation grant DMS-1712956.}
\runauthor{S. Minsker}

\affiliation{University of Southern California}
\address{
Stanislav Minsker \\
Department of Mathematics, \\
University of Southern California, \\
Los Angeles, CA 90089 \\
\printead{e1}
}
\end{aug}
\maketitle
\begin{abstract}
Estimation of the covariance matrix has attracted a lot of attention of the statistical research community over the years, partially due to important applications such as Principal Component Analysis. 
However, frequently used empirical covariance estimator, and its modifications, is very sensitive to the presence of outliers in the data. 
As P. Huber wrote \cite{huber1964robust}, ``...This raises a question which could have been asked already by Gauss, but which was, as far as I know, only raised a few years ago (notably by Tukey): what happens if the true distribution deviates slightly from the assumed normal one? As is now well known, the sample mean then may have a catastrophically bad performance...'' 
Motivated by Tukey's question, we develop a new estimator of the (element-wise) mean of a random matrix, which includes covariance estimation problem as a special case. 
Assuming that the entries of a matrix possess only finite second moment, this new estimator admits sub-Gaussian or sub-exponential concentration around the unknown mean in the operator norm. 
We explain the key ideas behind our construction, and discuss applications to covariance estimation and matrix completion problems.

\end{abstract}

\begin{keyword}[class=MSC]
\kwd[Primary ]{60B20, 62G35}
\kwd[; secondary ]{62H12}
\end{keyword}

\begin{keyword}
\kwd{random matrix}
\kwd{heavy tails}
\kwd{concentration inequality}
\kwd{covariance estimation}
\kwd{matrix completion.}
\end{keyword}

\end{frontmatter}

\section{Introduction}

Let $Y_1,\ldots,Y_n\in \mb C^{d_1\times d_2}$ be a sequence of independent random matrices such that all their entries have finite second moments: $\mb E\l|(Y_j)_{k,l}\r|^2<\infty$ for all $1\leq j\leq n, \ 1\leq k\leq d_1, \ 1\leq l\leq d_2$.  
Let $\mb EY_1,\ldots,\mb EY_n\in \mb C^{d_1\times d_2}$ be the expectations evaluated element-wise, meaning that $\l(\mb EY_j\r)_{k,l}=\mb E \l(Y_j\r)_{k,l}$. 
The goal of this paper is to construct and study estimators of 
$\mb E \bar Y:=\mb E\l[\frac{1}{n}\sum_{j=1}^n Y_j\r]$ under minimal assumptions on the distributions of $Y_1,\ldots,Y_n$. 
In particular, we are interested in the estimators that admit tight non-asymptotic bounds and exponential deviation inequalities without imposing any additional assumptions (besides finite second moments) on $Y_1,\ldots,Y_n$. 
For example, if $Y_j=Z_j Z_j^T$, where $Z_1,\ldots,Z_n\in \mb R^d$ are i.i.d. copies of a random vector $Z$ such that $\mb EZ=0$, $\mb E \l[ZZ^T\r]=\Sigma$ and $\mb E \|Z\|_2^4<\infty$, formulated problem is reduced to covariance estimation (here and in what follows, $\|\cdot\|_2$ and $\dotp{\cdot}{\cdot}$ stand for the usual Euclidean norm and Euclidean dot product respectively). 

Techniques developed in this paper have direct connection to several problems in high-dimensional statistics and statistical learning theory. 
In the past decade, these fields have seen numerous breakthroughs in structural estimation, concerned with a task of recovering a high-dimensional parameter that belongs to a set with ``simple'' structure from a small number of measurements. 
Examples include sparse linear regression, low-rank matrix recovery and structured covariance estimation. 
However, theoretical recovery guarantees for popular techniques (e.g., $\ell_1$ and nuclear norm minimization) usually require strong assumptions on the underlying probability distribution, such as sub-Gaussian or bounded noise. 
What happens with the performance of the algorithms when these conditions are violated, which is the case for many real data sets modeled by heavy-tailed distributions? Can the assumptions be weakened without sacrificing the quality of theoretical guarantees? 
We look at examples where the answer is positive, and describe modifications of existing techniques that allow to achieve the improvements.

\subsection{Overview of the previous work}

Let us begin by briefly discussing a scalar version of the problem investigated in this paper.  
Assume that $X_1,\ldots,X_n\in \mb R$ are i.i.d. copies of $X$, where $\mb EX^2<\infty$. 
One of the fundamental problems in statistics is to construct the confidence interval for the unknown mean $\mb EX$ based on a given sample. 
A surprising fact (dating back to \cite{Nemirovski1983Problem-complex00} where the ``median of means'' estimator was introduced, along with \cite{alon1996space} and \cite{jerrum1986random}) is that it's possible to construct a non-asymptotic confidence intervals $\hat I_n(\delta)$ with coverage probability $1-\delta$ (meaning that $\Pr\l(\mb EX\in \hat I_n(\delta)\r)\geq 1-\delta$ for given $n$ and $\delta$) and ``nearly optimal'' length 
$\l|\hat I_n(\delta)\r|\leq L\sqrt{\var(X)}\sqrt{\frac{\log(e/\delta)}{n}}$, where $L>0$ is an absolute constant. 
An in-depth study of this and closely related questions was performed in \cite{catoni2012challenging,devroye2015sub} based on two different approaches. 
Note that the center of any such confidence interval is a point estimator $\hat \mu:=\hat \mu(X_1,\ldots,X_n,\delta)$ that satisfies 
\[
\Pr\l( \l| \hat \mu - \mb EX \r|\geq L\sqrt{\var(X)}\sqrt{\frac{\log(e/\delta)}{n}} \r)\leq \delta.
\] 
Because the only assumption on $X$ is the existence of a second moment, it is natural to call such an estimator ``robust''
\footnote{For the classical treatment of robust estimators based on the notion of a breakdown point, we refer the reader to \cite{Huber2009Robust-statisti00}.}: 
it admits strong deviation bounds even for the heavy-tailed distributions that can be used to model outliers in the data. 
Ideas behind these results have also been extended to empirical risk minimization methods \cite{lerasle2011robust,brownlees2015empirical} which cover a wide range of statistical applications. 
Let us emphasize that the aforementioned estimators do not require any assumptions on the ``shape'' of the distribution, such as unimodality or elliptical symmetry. 
 
Generalizations of univariate results to the case of random vectors and random matrices are not straightforward since element-wise deviation inequalities do not always translate into desired bounds.
In some cases, element-wise bounds yield inequalities for the ``wrong'' norm: for example, estimating each entry of the covariance matrix results in a deviation inequality for the Frobenius norm, while we are frequently interested in the bounds for the operator norm that can be much smaller. 
An approach which often yields ``dimension - free'' bounds was proposed in \cite{hsu2013loss} and \cite{minsker2015geometric} (using generalizations of the median in higher dimensions); however, to the best of our knowledge, results of these papers are still not sufficient to obtain deviation guarantees in the operator norm that we are mainly interested in.
Under more restrictive assumptions on the sequence of random matrices $Y_1,\ldots,Y_n$ (such as $\|Y_j\|\leq M$ almost surely for some fixed $M>0$, $j=1,\ldots,n$, where $\|\cdot\|$ stands for the operator norm), behavior of the sample mean $\bar Y=\frac{1}{n}\sum_{j=1}^n Y_j$ has been analyzed with the help of matrix concentration inequalities \cite{Ahlswede2002Strong-converse00,oliveira2009concentration,tropp1}.

A closely related covariance matrix estimation problem has been extensively studied in the past decades. 
A comprehensive review is beyond the scope of this introduction, so we will just mention few classical results and more recent work related to the current line of research. 
Statistical properties of the sample covariance matrix for Gaussian and sub-Gaussian observations have been investigated in detail, see \cite{koltchinskii2017concentration,koltchinskii2016new,vershynin2010introduction,cai2010optimal,cai2016} and references therein; under weaker moment assumptions, sample covariance estimator has been studied in \cite{srivastava2013covariance}. 
Some popular robust estimators of scatter are discussed in \cite{hubert2008high}, including the Minimum Covariance
Determinant (MCD) estimator and the Minimum Volume Ellipsoid estimator (MVE). 
However, rigorous results for these estimators are available only for elliptically symmetric distributions; see \cite{butler1993asymptotics} for results on MCD and \cite{davies1992asymptotics} for results on MVE. 
Popular Maronna's \cite{maronna1976} and Tyler's \cite{tyler1987distribution,zhang2016marvcenko} M-estimators of scatter also admit theoretical guarantees for the family of elliptically symmetric distributions, but we are unaware of any results extending beyond this case.

Recent papers of O. Catoni \cite{catoni2016pac} and I. Guilini \cite{giulini2015pac},  Fan et al. \cite{fan2016eigenvector} are closest in spirit to our work. 
For instance, in \cite{catoni2016pac} author constructs a robust estimator of the Gram matrix of a random vector $Z\in \mb R^d$ (as well as its covariance matrix) via estimating the quadratic form $\mb E \dotp{Z}{u}^2$ uniformly over $\|u\|_2=1$, and obtains error bounds for the operator norm. 
The latter (univariate) estimators for the quadratic form are based on the fruitful ideas originating in \cite{catoni2012challenging}. 
However, results of these works can not be straightforwardly extended beyond covariance estimation, and are obtained under more stringent (compared to the present paper) assumptions on the underlying distribution (such as a known upper bound on the kurtosis of $\dotp{Z}{u}^2$ for any $u$ of norm $1$). 
In \cite{fan2016eigenvector}, authors obtain error bounds for norms other than the operator norm which is the main focus of the present paper. 

Finally, let us mention that the problem of robust matrix recovery (that is discussed as an example below) has also received attention recently: for instance, the work \cite{candes2011robust,klopp2014robust} investigates robust matrix completion under the ``low rank + sparse'' model. 
In \cite{fan2016robust}, authors study low-rank matrix recovery under the assumption that the additive noise has only $(2+\eps)$ moments, and obtain strong results via truncation argument. 
We propose a different approach based on general techniques developed in this paper and achieve similar results for the matrix completion problem while requiring only the finite variance of the noise.

\subsection{Organization of the paper}
\label{section:organization}

Section \ref{section:prelim} contains definitions, notation and background material. 
Our main results are introduced in section \ref{section:main}.  
After presenting core results, we discuss applications to covariance estimation and low-rank matrix completion in section \ref{section:examples}, and illustrate the role of various quantities involved in the general bounds through these examples. Sections \ref{section:variance} and \ref{section:shift} discuss adaptation to unknown parameters that appear in our construction, and contain longer proofs. 

Appendix contains proofs of several technical lemmas and results that were omitted in the main text.

\section{Preliminaries}
\label{section:prelim}

In this section, we introduce main notation and recall several useful facts from linear algebra, matrix analysis and probability theory that we rely on in the subsequent exposition. 

\subsection{Definitions and notation}

Given $A\in \mb C^{d_1\times d_2}$, let $A^\ast\in \mb C^{d_2\times d_1}$ be the Hermitian adjoint of $A$. 
If $A$ is self-adjoint, we will write $\lambda_{\mx}(A)$ and $\lambda_{\mn}(A)$ for the largest and smallest eigenvalues of $A$. 
Next, we will introduce the matrix norms used in the paper. 

Everywhere below, $\|\cdot\|$ stands for the operator norm $\|A\|:=\sqrt{\lambda_{\mx}(A^\ast A)}$. 
If $d_1=d_2=d$, we denote by $\tr A$ the trace of $A$.
Next, for $A\in \mb C^{d_1\times d_2}$, the nuclear norm $\|\cdot\|_1$ is defined as 
$\|A\|_1=\tr(\sqrt{A^*A})$, where $\sqrt{A^*A}$ is a nonnegative definite matrix such that $(\sqrt{A^*A})^2=A^\ast A$. 
The Frobenius (or Hilbert-Schmidt) norm is $\|A\|_{\mathrm{F}}=\sqrt{\tr(A^\ast A)}$, and the associated inner product is 
$\dotp{A_1}{A_2}=\tr(A_1^\ast A_2)$. 
Finally, set $\|A\|_{\max}:=\sup_{i,j}|a_{i,j}|$. 
For $Y\in \mb C^d$, $\l\| Y \r\|_2$ stands for the usual Euclidean norm of $Y$. 

Given two self-adjoint matrices $A$ and $B$, we will write $A\succeq B \ (\text{or }A\succ B)$ iff $A-B$ is nonnegative (or positive) definite.

Given a sequence $Y_1,\ldots,Y_n$ of random matrices, $\mb E_j[\,\cdot \,]$ will stand for the conditional expectation 
$\mb E[\,\cdot\,|Y_1,\ldots,Y_{j}]$.

Finally, for $a,b\in \mb R$, set $a\vee b:=\max(a,b)$ and $a\wedge b:=\min(a,b)$. 

\subsection{Tools from linear algebra}

In this section, we collect several facts from linear algebra, matrix analysis and probability theory that are frequently used in our arguments. 

\begin{definition}
\label{matrix-function}
Given a real-valued function $f$ defined on an interval $\mb T\subseteq \mb R$ and a self-adjoint $A\in \mb C^{d\times d}$ with the eigenvalue decomposition 
$A=U\Lambda U^\ast$ such that $\lambda_j(A)\in \mb T,\ j=1,\ldots,d$, define $f(A)$ as 
$f(A)=Uf(\Lambda) U^\ast$, where 
\[
f(\Lambda)=f\l( \begin{pmatrix}
\lambda_1 & \,  & \,\\
\, & \ddots & \, \\
\, & \, & \lambda_d
\end{pmatrix} \r)
=\begin{pmatrix}
f(\lambda_1) & \,  & \,\\
\, & \ddots & \, \\
\, & \, & f(\lambda_d)
\end{pmatrix}.
\] 
\end{definition}
\noindent 
Additionally, we will often use the following facts: 
\begin{fact}
\label{fact:01}
Let $A\in \mb C^{d\times d}$ be a self-adjoint matrix, and $f_1, \ f_2$ be two real-valued functions such that $f_1(\lambda_j)\geq f_2(\lambda_j)$ for $j=1,\ldots,d$. 
Then $f_1(A)\succeq f_2(A)$. 
\end{fact}

\begin{fact}
\label{fact:02}
Let $A,B\in \mb C^{d\times d}$ be two self-adjoint matrices such that $A\succeq B$. 
Then $\lambda_j(A)\geq \lambda_j(B), \ j=1,\ldots, d$, where $\lambda_j(\cdot)$ stands for the $j$-th largest eigenvalue.   
Moreover, $\tr e^A\geq \tr e^B$.
\end{fact}

\begin{fact}
\label{fact:03}
Matrix logarithm is operator monotone: if $A\succ 0, \ B\succ 0$ and $A\succeq B$, then $\log(A)\succeq \log(B)$. 
\end{fact}
\begin{proof}
See \cite{bhatia1997matrix}. 
\end{proof}

\begin{fact}
\label{fact:04}
Let $A\in \mb C^{d\times d}$ be a self-adjoint matrix. 
Then $I+A+\frac{A^2}{2}\succ 0$. 
Moreover, 
\[
-\log \l( I+A+\frac{A^2}{2} \r) \preceq \log\l( I-A+\frac{A^2}{2} \r).
\] 
\end{fact}
\begin{proof}
In view of the definition of a matrix function, the first claim follows from scalar inequality $1+t+t^2/2>0$ for $t\in\mb R$. 
Similarly, the second relation follows from the inequality $-\log(1+t+t^2/2)\leq \log(1-t+t^2/2)$ for $t\in\mb R$.
\end{proof}

\begin{fact}[Lieb's concavity theorem]
\label{fact:05}
Given a fixed self-adjoint matrix $H$, the function
\[
A\mapsto \tr\exp\l(H+\log(A)\r)
\]
is concave on the cone of positive definite matrices.
\end{fact}
\begin{proof}
See \cite{lieb1} and \cite{tropp2015introduction}.
\footnote{Let us mention that Lieb's theorem is one of the key tools for proving matrix concentration inequalities, and its power in this context was first demonstrated by J. Tropp \cite{tropp2012user}.}
\end{proof}

\begin{fact}
\label{fact:06}
Let $f:\mb R\mapsto \mb R$ be a convex function. 
Then $A\mapsto \tr f(A)$ is convex on the set of self-adjoint matrices. 
In particular, for any self-adjoint matrices $A,B$, 
\[
\tr f\l( \frac{A+B}{2} \r)\leq \frac{1}{2}\tr f(A) + \frac{1}{2}\tr f(B).
\] 
\end{fact}
\begin{proof}
This is a consequence of Peierls inequality, see Theorem 2.9 in \cite{carlen2010trace} and the comments following it. 
\end{proof}

Finally, we introduce the Hermitian dilation which allows to reduce many problems involving general rectangular matrices to the case of Hermitian operators. 
Given the rectangular matrix $A\in\mb C^{d_1\times d_2}$, the Hermitian dilation $\m H: \mb C^{d_1\times d_2}\mapsto \mb C^{(d_1+d_2)\times (d_1+d_2)}$ is defined as
\begin{align}
\label{eq:dilation}
&
\m H(A)=\begin{pmatrix}
0 & A \\
A^\ast & 0
\end{pmatrix}.
\end{align}
Since 
$\m H(A)^2=\begin{pmatrix}
A A^\ast & 0 \\
0 & A^\ast A
\end{pmatrix},$ 
it is easy to see that $\| \m H(A) \|=\|A\|$. 
Another tool useful in dealing with rectangular matrices is the following lemma:
\begin{lemma}
\label{lemma:dilation}
Let $S\in \mb C^{d_1\times d_1}, \ T\in \mb C^{d_2\times d_2}$ be self-adjoint matrices, and $A\in \mb C^{d_1\times d_2}$. 
Then 
\[
\l\| \begin{pmatrix}
S & A \\
A^\ast & T
\end{pmatrix}\r\|
\geq
\l\| \begin{pmatrix}
0 & A \\
A^\ast & 0
\end{pmatrix} \r\|.
\]
\end{lemma}
\begin{proof}
See section \ref{proof:dilation} in the appendix.
\end{proof}

\section{Main results}
\label{section:main}

Our construction has its roots in the technique proposed by O. Catoni \cite{catoni2012challenging} for estimating the univariate mean. 
Let us briefly recall the main ideas of Catoni's approach. 
Assume that $\xi,\xi_1,\ldots,\xi_n$ is a sequence of i.i.d. random variables such that $\mb E\xi=\mu$ and $\var(\xi)\leq v^2$. 
Catoni's estimator is defined as follows: 
let $\psi(x):\mb R\mapsto \mb R$ be a non-decreasing function such that for all $x\in \mb R$,
\begin{align}
\label{eq:psi}
-\log(1-x+x^2/2)\leq \psi(x) \leq \log(1+x+x^2/2).
\end{align}
See remark \ref{remark:psi2} below for examples of such functions. 
Given $\theta>0$, let $\hat \mu_{\theta}$ be such that 
\begin{align}
\label{eq:catoni}
\sum_{j=1}^n \psi\l( \theta(\xi_j-\hat\mu_{\theta})  \r)=0
\end{align}
(clearly, $\hat\mu_\theta$ always exists due to monotonicity). 
Set $\eta=v\sqrt{\frac{2 t}{n(1-2t/n)}}$ and $\theta_\ast=\sqrt{\frac{2t}{n(v^2+\eta^2)}}$. 
Assuming that $n>2t$, it is shown in \cite{catoni2012challenging} that
$
|\hat\mu_{\theta_\ast} - \mu|\leq \eta
$
with probability $\geq 1-2e^{-t}$.

We proceed by presenting a multivariate extension of the estimator $\hat \mu_\theta$. 
We will first formulate main results for the self-adjoint matrices, and will later deduce the general case of rectangular matrices as a corollary. 
Let $Y_1,\ldots,Y_n\in \mb C^{d\times d}$ be a sequence of independent self-adjoint random matrices 
such that $\sigma_n^2:=\l\|  \sum\limits_{j=1}^n\mb EY_j^2  \r\|<\infty$. 
Let $\Psi$ be such that $\Psi'(x)=\psi(x)$ for all $x\in \mb R$, and set 
\begin{align}
\label{eq:matrix-catoni}
\wh T_\theta^\ast = \argmin_{S\in \mb C^{d\times d},S=S^\ast}\l[ \tr \sum_{j=1}^n \Psi\l( \theta(Y_j - S) \r) \r], 
\end{align}
where $\theta>0$ is an appropriate constant. 
It follows from Fact \ref{fact:06} that $\wh T^\ast_\theta$ exists, moreover, it is unique if $\psi(x)$ is strictly increasing. 
It is also not hard to see that \eqref{eq:matrix-catoni} is equivalent to 
\begin{align}
\label{eq:matrix-catoni2}
\sum_{j=1}^n \psi \l( \theta(Y_j - \wh T^\ast_\theta)\r) = 0_{d\times d}.
\end{align}
Indeed, if $F_\psi(S):= \tr \sum_{j=1}^n \Psi\l( \theta(Y_j - S) \r)$, then \eqref{eq:matrix-catoni2} simply states that the gradient of $F_\psi$ evaluated at $\wh T^\ast_\theta$ is equal to zero; see Lemma \ref{lemma:gradient} in the appendix for more details.

To understand the properties of the estimator defined via \eqref{eq:matrix-catoni} and \eqref{eq:matrix-catoni2}, we will first consider another estimator $\wh T^{(0)}_\theta$ that shares many important properties with $\wh T^\ast_\theta$ but is easier to analyze.

The ``preliminary estimator'' $\wh T^{(0)}_\theta$ is constructed as follows: given $\theta>0$ and a function 
$\psi$ satisfying (\ref{eq:psi}), set $X_j:=\psi\l( \theta Y_j \r), \ j=1,\ldots,n$ and 
\begin{align}
\label{eq:T^0}
\wh T^{(0)}_\theta := \frac{1}{n\theta}\sum_{j=1}^n X_j.
\end{align}
In other words, $\wh T^{(0)}_\theta$ is an average of ``$\psi$-truncated'' observations. 
Since $X_j\simeq \theta Y_j$ for small $\theta$ and a smooth function $\psi$, we expect that $\wh T_\theta^{(0)}$ is close to $\frac{1}{n}\sum_{j=1}^n \mb EY_j$.  
In the following sections, we will make this intuition more precise. 
In particular, we will establish the following (so far informally stated) results: 
\begin{thm} 
\begin{asparaenum}
\item
Assume that the observations $Y_1,\ldots,Y_n$ are i.i.d. copies of $Y\in \mb C^{d\times d}$ and the parameter $\theta$ is chosen properly. Then 
\[
\Pr\l( \l\| \wh T_\theta^{(0)}- \mb EY \r\| \geq \sigma\sqrt{\frac{t}{n}} \r)\leq 2d \exp\l( -\frac{t}{2} \r),
\] 
where $\sigma^2:=\sigma_n^2/n=\l\| \mb E Y^2\r\|$. 
\item Assume that $n$ is large enough and $\theta$ is chosen properly. Then the estimator $\wh T_\theta^\ast$ defined via \eqref{eq:matrix-catoni2} satisfies the inequality 
\[
\Pr\l( \l\| \wh T_\theta^\ast - \mb EY \r\| \geq C_1\sigma_0\sqrt{\frac{t}{n}} \r) \leq C_2 d \exp\l( -\frac{t}{2} \r),
\]
where $C_1,C_2>0$ are absolute constants and $\sigma_0^2:=\l\| \mb E(Y-\mb EY)^2 \r\|$. 
\end{asparaenum}
\end{thm}
Note that the ``variance term'' $\l\| \mb EY^2 \r\|$ appearing in the first part of the bound above is akin to the second moment, while in the second bound it is replaced by $\sigma_0^2=\l\| \mb E(Y-\mb EY)^2 \r\|$; presence of the term $\l\| \mb EY^2 \r\|$ can be explained by the fact that the estimator $\wh T^{(0)}_\theta$ is obtained via bias-producing truncation. 
We remark that in some applications, such as matrix completion discussed in section \ref{section:examples}, even the estimator $\wh T_\theta^{(0)}$ with ``suboptimal'' variance term suffices to obtain good bounds.

\begin{remark}
\label{remark:psi2}
Most of our results do not depend on the concrete choice of the function $\psi$. 
One possibility is 
\begin{align}
\label{eq:psi2}
\psi_1(x)=\begin{cases}
\log\l(1+x+\frac{x^2}{2}\r), & x\geq 0, \\
-\log\l(1-x+\frac{x^2}{2}\r), & x<0.
\end{cases}
\end{align}
Another example is 
\begin{align}
\label{eq:psi3}
\psi_2(x)=\begin{cases}
x - \sign(x)\frac{x^2}{2}, & x\in[-1,1], \\
\frac{1}{2}\sign(x), & |x|>1.
\end{cases}
\end{align}
Since the latter function is bounded, it can provide additional advantages (such as robustness) in applications. 
However, note that $\psi_2(x)$ does not satisfy \eqref{eq:psi}; instead, it satisfies a slightly weaker inequality 
\[
-\log\l(1-x+x^2\r)\leq \psi_2(x) \leq \log\l(1+x+x^2 \r),
\]
hence all subsequent results hold for $\psi_2$ as well, albeit with slightly worse constant factors. 
We also note that both $\psi_1$ and $\psi_2$ are operator Lipschitz functions; see Lemma \ref{lemma:lipschitz} for details. 
\end{remark}

\subsection{Bounds for the moment generating function}

In this section, we will establish deviation inequalities for the estimator $\wh T_\theta^{(0)}=\frac{1}{n\theta}\sum_{j=1}^n \psi(\theta Y_j)$. 
The lemma below is the cornerstone of our results. 
As before, given $\theta>0$, let $X_j=\psi(\theta Y_j)$.
\begin{lemma}
\label{lemma:main}
The following inequalities hold:
\begin{align}
&
\label{main:01}
\mb E \tr \exp\l( \sum_{j=1}^n \l( X_j -\theta\mb EY_j\r) \r)\leq \tr \exp\l( \frac{\theta^2}{2}\sum_{j=1}^n \mb EY_j^2 \r), \\
&
\label{main:02}
\mb E \tr \exp\l( \sum_{j=1}^n \l( \theta\mb EY_j - X_j\r) \r)\leq \tr \exp\l( \frac{\theta^2}{2}\sum_{j=1}^n \mb EY_j^2 \r).
\end{align}
\end{lemma}
\begin{proof}
Note that 
\begin{align*}
\mb E\, &\tr \exp\l( \sum_{j=1}^n \l( X_j -\theta\mb EY_j\r) \r) = \\ 
&\mb E \mb E_{n-1} \tr \exp\l( \l[ \sum_{j=1}^{n-1} \l( X_j -\theta\mb EY_j\r) -\theta\mb EY_n \r] + \psi\l( \theta Y_n  \r) \r)\leq \\
&\mb E \mb E_{n-1} \tr \exp\l( \l[ \sum_{j=1}^{n-1} \l( X_j -\theta\mb EY_j\r) -\theta\mb EY_n \r] + \log\l(I+\theta Y_n +\theta^2 Y_n^2/2\r) \r)\leq \\
&
\mb E\, \tr \exp\l( \sum_{j=1}^{n-1} \l( X_j -\theta\mb EY_j\r)  + \log\l( I+\theta \mb EY_n +\theta^2 \mb EY_n^2/2 \r) - \theta\mb EY_n \r),
\end{align*}
where the first inequality follows from the semidefinite relation $\psi(\theta Y_n)\preceq \log\l( I +\theta Y_n+\frac{\theta^2}{2}Y_n^2\r)$ and Fact \ref{fact:02}, 
and the second inequality follows from Lieb's concavity theorem (Fact \ref{fact:05}) with 
$
H= \sum_{j=1}^{n-1} \l( X_j -\theta\mb EY_j\r) - \theta\mb EY_n
$ 
and Jensen's inequality for conditional expectation. 
We also note that $I+\theta\mb EY_n+\theta^2\mb EY_n^2/2\succ 0$ since $I+\theta Y_n+\theta^2 Y_n^2/2\succ 0$ almost surely, hence $\log\l( I+\theta \mb EY_n +\theta^2 \mb EY_n^2/2 \r)$ is well-defined.
Repeating the steps for $X_{n-1},\ldots,X_1$, we obtain the inequality
\begin{align}
\label{eq:01}
&
\mb E \, \tr \exp\l( \sum_{j=1}^n \l( X_j - \theta\mb EY_j\r) \r) \leq 
\mb \tr \exp\l( \sum_{j=1}^n \l( \log\l( I+\theta \mb EY_j +\theta^2 \mb EY_j^2/2 \r) - \theta\mb EY_j \r) \r)
\end{align}
It remains to note that by Fact \ref{fact:01} and the inequality $\log(1+x)\leq x$ (that holds $\forall \ x>-1$), for all $j=1,\ldots, n$
\[
\log\l( I+\theta \mb E Y_j + \theta^2 \mb E Y_j^2/2 \r) \preceq \theta\mb EY_j + \frac{\theta^2}{2}\mb E Y_j^2,
\]
or $\log\l( I+\theta \mb E Y_j + \theta^2 \mb E Y_j^2/2 \r)- \theta\mb EY_j \preceq \frac{\theta^2}{2}\mb E Y_j^2$. 
The first inequality (\ref{main:01}) now follows from (\ref{eq:01}) and Fact \ref{fact:02}.

To establish the second inequality of the lemma, we use the relation 
$-X_j=-\psi(\theta Y_j)\preceq \log\l( I -\theta Y_j + \frac{\theta^2}{2}Y_j^2\r)$ (which follows from (\ref{eq:psi}) and Fact \ref{fact:01})
together with the Fact \ref{fact:02} to deduce that 
\begin{align*}
\mb E\, \tr \exp&\l( \sum_{j=1}^n \l( \theta\mb EY_j - X_j\r) \r)\leq \\
&
\mb E\, \tr \exp\l(  \sum_{j=1}^n \l( \log\l(I +\theta (-Y_j)+\theta^2 Y_j^2/2 \r) - \theta\mb E(-Y_j)\r)  \r),
\end{align*}
and apply inequality (\ref{main:01}) to the sequence $-Y_1,\ldots,-Y_n$ with 
\[
X_j=\log\l(I +\theta (-Y_j)+\theta^2 (-Y_j)^2/2 \r), \  j=1,\ldots,n.
\] 
\end{proof}
\noindent We are ready to state and prove the main result of this section. 
\begin{theorem}
\label{th:main}
Let $Y_1,\ldots,Y_n\in \mb C^{d\times d}$ be a sequence of independent self-adjoint random matrices, and 
$\sigma_n^2\geq\l\| \sum_{j=1}^n \mb EY_j^2 \r\|$. 
Then for all $\theta>0$
\[
\Pr\l( \l\| \sum_{j=1}^n \l(\frac{1}{\theta}\psi\l( \theta Y_j  \r) - \mb EY_j \r) \r\| \geq t\sqrt n \r)
\leq 2d \exp\l( -\theta t\sqrt{n} + \frac{\theta^2\sigma_n^2}{2} \r).
\]
In particular, setting $\theta=\frac{t\sqrt{n}}{\sigma_n^2}$, we get the ``sub-Gaussian'' tail bound $2d \exp\l( -\frac{t^2}{2\sigma_n^2/n} \r)$, for a given $t>0$. 
Alternatively, setting $\theta=\frac{\sqrt{n}}{\sigma_n^2}$ (independent of $t$), we obtain sub-exponential concentration with tail 
$2d \exp\l( -\frac{2t-1}{2\sigma_n^2/n} \r)$ for all $t>1/2$. 
\end{theorem}
\begin{remark}
\label{remark:iid}
In the important special case when $Y_j, \ j=1,\ldots,n$ are i.i.d. copies of $Y$, we will often use the following equivalent form of of the bound: assume that $\sigma^2\geq\| \mb EY^2 \|$, then replacing $t$ by $\sigma \sqrt{s}$ and setting $\theta:=\sqrt{\frac{s}{n}}\frac{1}{\sigma}$ implies that 
\begin{align}
\label{th:iid}
&
\Pr\l( \l\| \wh T_\theta^{(0)} - \mb EY \r\| \geq \sigma \sqrt{\frac{s}{n}} \r)
\leq 2d \exp\l( -s/2 \r),
\end{align}
where $\wh T_\theta^{(0)}$ was defined in \eqref{eq:T^0}.
\end{remark}
\begin{proof}
As before, set 
$X_j:=\psi\l( \theta Y_j \r), \ j=1,\ldots,n$. 
Then 
\begin{align*}
\Pr\Bigg( \lambda_{\mx}&\l( \frac{1}{\theta}\sum_{j=1}^n \l(X_j - \theta\mb E Y_j\r) \r) \geq s \Bigg)  \\
&
=\Pr\l( \exp\l( \lambda_{\mx}\l( \sum_{j=1}^n \l(X_j - \theta\mb E Y_j\r) \r)  \r) \geq e^{\theta s} \r) \\
&
\leq e^{-\theta s}\mb E\, \tr \exp\l( \sum_{j=1}^n \l( X_j -\theta\mb EY_j\r) \r) 
\leq e^{-\theta s} \, \tr \exp\l( \frac{\theta^2}{2}\sum_{j=1}^n \mb EY_j^2 \r)  \\
&
\leq d\exp\l( -\theta s + \frac{\theta^2}{2}\l\| \sum_{j=1}^n \mb EY_j^2 \r\|  \r),
\end{align*}
where we used Chebyshev's inequality, the fact that $e^{\lambda_{\mx}(A)}=\lambda_{\mx}(e^A)$ and the inequality $\lambda_{\mx} (e^A)\leq \tr e^A$ on the second step, the first inequality of Lemma \ref{lemma:main} on the third step, and the bound 
$\tr e^A\leq d \, e^{\|A\|}$ on the last step (here and below, $A\in \mb C^{d\times d}$ is an arbitrary self-adjoint matrix). 
Similarly, since $-\lambda_{\mn}(A)=\lambda_{\mx}(-A)$, we have 
\begin{align*}
\Pr\Bigg( \lambda_{\mn} & \l( \frac{1}{\theta}\sum_{j=1}^n \l(X_j - \theta\mb E Y_j\r) \r) \leq -s\Bigg) \\
& = \Pr\l( \lambda_{\mx}\l( \frac{1}{\theta}\sum_{j=1}^n \l(\theta\mb EY_j - X_j\r) \r) \geq s\r) \\
&
\leq e^{-\theta s}\mb E \tr \exp\l( \sum_{j=1}^n \l(\theta\mb EY_j - X_j\r) \r)
\leq e^{-\theta s} \, \tr \exp\l( \frac{\theta^2}{2}\sum_{j=1}^n \mb EY_j^2 \r)  \\
&
\leq d\exp\l( -\theta s + \frac{\theta^2}{2}\l\| \sum_{j=1}^n \mb EY_j^2 \r\|  \r),
\end{align*}
where we used the second inequality of Lemma \ref{lemma:main} instead. 
The result follows by taking $s:=t\sqrt{n}$ since for a self-adjoint matrix $A$, $\|A\|=\max\l( \r. \lambda_{\mx}(A),$
$ -\lambda_{\mn}(A) \l.\r)$. 
\end{proof}

The main weakness of the estimator $\wh T_\theta^{0}$ discussed above is the fact that the ``variance term'' $\l\| \sum_{j=1}^n \mb EY_j^2 \r\|$ appearing in the bound is akin to the second moment (the price we pay for applying bias-producing truncation) while we would like to replace it by $\l\| \sum_{j=1}^n \mb E(Y_j-\mb EY_j)^2 \r\|$. 
This problem will be addressed in detail in section \ref{section:shift}. 
In particular, we will show the following: assume that $Y_1,\ldots,Y_n$ are i.i.d. copies of $Y$, $\sigma_0^2 \geq \l\| \mb E(Y-\mb EY)^2 \r\|$, 
$\theta_0=\sqrt{\frac{2t}{n}}\frac{1}{\sigma_0}$, and $n$ is large enough ($n\gtrsim d^2$). 
Then, with exponentially high probability with respect to $s$, the solution $\wh T_{\theta_0}^\ast$ of equation \eqref{eq:matrix-catoni2} satisfies $\l\| \wh T_{\theta_0}^\ast - \mb EY \r\|\leq C\sigma_0\sqrt{\frac{s}{n}}$ for an absolute constant $C>0$. 
Another problem is the fact that one needs to know the value of $\l\| \sum_{j=1}^n \mb EY_j^2 \r\|$ (or its tight upper bound) a priori to choose the ``optimal'' value of parameter $\theta$. 
This issue and its resolution based on adaptive estimators  is discussed in section \ref{section:variance}. 
We conclude this discussion with few additional comments.
\begin{remark}
\begin{asparaenum}
\item 
Sub-Gaussian guarantees provided by Theorem \ref{th:main} hold for a given confidence parameter $t>0$ that has to be fixed a priori: in particular, the optimal value of $\theta$ depends it. 
However, as it was noted in \cite{devroye2015sub}, this is sufficient to construct (via Lepski's method \cite{lepskii1992asymptotically}) estimators that admit sub-Gaussian tails uniformly over $t$ in a certain range. 
\item 
Let $Y_1,\ldots,Y_n\in \mb C^{d\times d}$ be i.i.d. copies of $Y$, and $\sigma_0^2=\| \mb E(Y - \mb EY)^2 \|$. 
It is interesting to compare our estimator (in particular, bound (\ref{th:iid})) to the guarantees for the sample mean $\frac{1}{n}\sum_{j=1}^n Y_j$. 
Under an additional restrictive boundedness assumption requiring that $\|Y\|\leq M$ almost surely, 
the noncommutative Bernstein's inequality (see Theorem 1.4 in \cite{tropp2012user}) implies that
$
\l\| \frac{1}{n}\sum_{j=1}^n Y_j -\mb EY \r\|\leq 2\sigma_0 \sqrt{\frac{t}{n}}\vee \frac{4}{3} \frac{M t}{n}
$
with probability $\geq 1-2d e^{-t/2}$. 
Hence, even under additional strong assumptions our technique allows to obtain guarantees that compare favorably to the sample mean. 
However, as noted in \cite{tropp2012user}, in the case when $\|Y\|\leq M$ almost surely, the size of 
$\mb E \l\| \frac{1}{n}\sum_{j=1}^n Y_j -\mb EY \r\|$ is controlled by $\sigma_0^2$ while the scale of deviations of the random variable
$
\l| \l\| \frac{1}{n}\sum_{j=1}^n Y_j -\mb EY \r\| - \mb E \l\| \frac{1}{n}\sum_{j=1}^n Y_j -\mb EY \r\|\r|
$
depends on the ``weak variance'' parameter 
$\sigma^2_\ast = \sup_{\|v\|_2=1}\mb E \dotp{(Y - \mb EY) v}{v}^2\leq \sigma_0^2$. It is not clear if similar improvements are achievable in the case of heavy-tailed distributions; see Remark \ref{remark:sigma} for additional comments.
\end{asparaenum}
\end{remark}

\subsection{Bounds depending on the effective dimension}

The bound obtained in Theorem \ref{th:main} explicitly depends on the dimension $d$ of random matrices.  
Example is subsection \ref{ex:dimension} below shows that the dimensional factor in the right-hand side of the inequality is unavoidable in general. However, it is possible to prove a similar inequality which only includes the ``effective dimension'' defined as 
\begin{align}
\label{effdim}
\bar d:=\frac{\tr \l( \sum_{j=1}^n \mb EY_j^2\r) }{\l\| \sum_{j=1}^n \mb EY_j^2  \r\|}
\end{align}
which can be much smaller than $d$ if $\sum_{j=1}^n \mb EY_j^2$ has many eigenvalues that are close to $0$. 
The following result holds:
\begin{theorem}
\label{th:intdim}
Let $Y_1,\ldots,Y_n\in \mb C^{d\times d}$ be a sequence of independent self-adjoint random matrices, and 
$\sigma_n^2\geq\l\| \sum_{j=1}^n \mb EY_j^2 \r\|.$ 
Then 
\[
\Pr\l( \l\| \sum_{j=1}^n \l(\frac{1}{\theta}\psi\l( \theta Y_j \r) - \mb EY_j \r) \r\| \geq t\sqrt n \r)
\leq 2 \bar d \l( 1 + \frac{1}{\theta t \sqrt n} \r) \exp\l( -\theta t\sqrt{n} + \frac{\theta^2\sigma_n^2}{2} \r).
\]
\end{theorem}
\begin{remark}
As before, we can set $\theta=\frac{t\sqrt{n}}{\sigma_n^2}$ to get 
\[
\Pr\l( \l\| \sum_{j=1}^n \l(\frac{1}{\theta}\psi\l( \theta Y_j \r) - \mb EY_j \r) \r\| \geq t\sqrt n \r)
\leq 2 \bar d \l( 1 + \frac{\sigma_n^2/n}{t^2} \r) \exp\l( -\frac{t^2}{2\sigma_n^2/n} \r).
\]
For the values of $t\geq \sqrt{\sigma_n^2/n}$ (when the bound becomes useful), it further simplifies to  
\[
\Pr\l( \l\| \sum_{j=1}^n \l(\frac{1}{\theta}\psi\l( \theta Y_j \r) - \mb EY_j \r) \r\| \geq t\sqrt n \r)
\leq 4 \bar d \exp\l( -\frac{t^2}{2\sigma_n^2/n} \r).
\] 
For the ``sub-exponential regime'' with $\theta=\frac{\sqrt{n}}{\sigma_n^2}$, we get that for all $t\geq \frac{1}{2}\vee \sigma_n^2/n$ simultaneously,
\[
\Pr\l( \l\| \sum_{j=1}^n \l(\frac{1}{\theta}\psi\l( \theta Y_j \r) - \mb EY_j \r) \r\| \geq t\sqrt n \r)
\leq 4 \bar d \exp\l( -\frac{2t-1}{2\sigma_n^2/n} \r).
\]
\end{remark}

\begin{proof}
The argument is similar in spirit to the proof of Theorem \ref{th:main}. 
Details are included in appendix \ref{proof:intdim}.
\end{proof}

\subsubsection{Dimensional factor in Theorem \ref{th:main}}
\label{ex:dimension}
Example below shows that the dimensional factor in Theorem \ref{th:main} is unavoidable in general. 
Assume that $\psi(x)=\psi_1(x)$ as defined in \eqref{eq:psi2}, $\theta=1$, $n=d$, and let $Y_j, \ j\leq d$ be independent and such that 
$\psi_1(Y_j)=\gamma_j \, e_j e_j^T$, where $\gamma_j, \ j\leq d$ are i.i.d. random variables with density $p(x) = e^{- 2|x|}$, and $\{e_1,\ldots,e_d\}$ is the standard Euclidean basis. 
Recalling that $Y_j = \psi_1^{-1}(\gamma_j)e_j e_j^T$, it is easy to check that $\mb EY_j = 0_{d\times d}$, and that 
$\l\| \sum_{j=1}^d \mb EY_j^2 \r\| = \mb E \l( \psi_1^{-1}(\gamma_1) \r)^2<\infty$. 
Theorem \ref{th:main} implies that 
\[
\Pr\l( \l\| \sum_{j=1}^d \gamma_j  e_j e_j^T \r\| \geq s \r)\leq f(d) e^{-s}
\] 
with $f(d)\leq Cd$ for some absolute constant $C$.  
Since $\l\| \sum_{j=1}^d \gamma_j e_j e_j^T \r\| = \max\l( |\gamma_1|,\ldots,|\gamma_d| \r)$, it follows from Lemma \ref{lemma:lower} that 
\[
\Pr\l(  \l\|  \sum_{j=1}^d \gamma_j \, e_j e_j^T  \r\| \geq \l(\frac{1}{2} - \tau \r)\log d  \r) \geq c(\tau)
\] 
for any $0<\tau<1/2$ and some constant $c(\tau)>0$. 
This shows that the dimensional factor $f(d)$ can not grow slower than $d^{1/2-\tau}$ for any $\tau>0$.

\subsection{Bounds for arbitrary rectangular matrices}

In this section, we will deduce results for arbitrary matrices from the bounds for self-adjoint operators. 
Let $Y_1,\ldots,Y_n\in \mb C^{d_1\times d_2}$ be independent, and assume that 
\[
\sigma_n^2\geq \max\l( \l\|  \sum_{j=1}^n \mb E Y_j Y_j^\ast \r\|, \l\|  \sum_{j=1}^n \mb E Y_j^\ast Y_j \r\| \r).
\]
Given $\theta>0$, set 
$X_j:=\psi\l( \theta\m H(Y_j) \r)$ (where $\m H(\cdot)$ is the self-adjoint dilation, see equation (\ref{eq:dilation})) and 
define $\hat T\in \mb C^{(d_1+d_2)\times (d_1+d_2)}$ as
\begin{align*}
&
\widehat T:=\widehat T(\theta)=\sum_{j=1}^n\frac{1}{\theta} X_j.
\end{align*}
Let $\hat T_{11}\in \mb C^{d_1\times d_1}$, $\hat T_{22}\in \mb C^{d_2\times d_2}$, $\hat T_{12}\in \mb C^{d_1\times d_2}$ be such that 
$\widehat T=\begin{pmatrix}
\hat T_{11} & \hat T_{12} \\
\hat T^\ast_{12} & \hat T_{22}
\end{pmatrix}.$ 
Since $\widehat T$ is ``close'' to $\sum_{j=1}^n \m H\l( \mb EY_j \r)$ for the proper choice of $\theta$, it is natural to expect that 
$\hat T_{12}$ is close to $\sum_{j=1}^n \mb EY_j$. 
\begin{corollary}
\label{cor:rectangular}
Under the assumptions stated above,
\[
\Pr\l( \l\| \hat T_{12} - \sum_{j=1}^n \mb EY_j  \r\| \geq t\sqrt n \r)
\leq 2(d_1+d_2) \exp\l( -\theta t\sqrt{n} + \frac{\theta^2\sigma_n^2}{2} \r)
\]
and
\[
\Pr\l( \l\| \hat T_{12} - \sum_{j=1}^n \mb EY_j  \r\| \geq t\sqrt n \r)
\leq 2\bar d\l( 1+\frac{1}{\theta t\sqrt{n}} \r) \exp\l( -\theta t\sqrt{n} + \frac{\theta^2\sigma_n^2}{2} \r),
\]
where 
$\bar d=2 \frac{  \tr\l(\sum_{j=1}^n \mb E Y_j^\ast Y_j\r) }{ \l\|  \sum_{j=1}^n \mb E Y_j Y_j^\ast \r\| \vee \l\|  \sum_{j=1}^n \mb E Y_j^\ast Y_j \r\| }$.
\end{corollary}
\begin{proof}
Note that 
\[
\l\| \sum_{j=1}^n \mb E\,\m H(Y_j)^2 \r\| = \max\l( \l\|  \sum_{j=1}^n \mb E Y_j Y_j^\ast \r\|, \l\|  \sum_{j=1}^n \mb E Y_j^\ast Y_j \r\| \r)\leq \sigma_n^2.
\] 
Theorem \ref{th:main} applied to self-adjoint random matrices $\m H(Y_j)\in \mb C^{(d_1+d_2)\times (d_1+d_2)}, \ j=1,\ldots,n$ implies that 
$
\l\| \widehat T - \sum_{j=1}^n \m H(\mb E Y_j) \r\|\leq t\sqrt{n}
$
with probability $\geq 1-2(d_1+d_2) \exp\l( -\theta t\sqrt{n} + \frac{\theta^2\sigma_n^2}{2} \r)$. 
It remains to apply Lemma \ref{lemma:dilation}: 
\begin{align*}
\l\| \widehat T - \sum_{j=1}^n \m H(\mb E Y_j)\r\| & =
\l\| \begin{pmatrix}
\hat T_{11} & \hat T_{12} - \sum_{j=1}^n \mb E Y_j \\
\hat T^\ast_{12} - \sum_{j=1}^n \mb E Y^\ast_j & \hat T_{22}
\end{pmatrix} \r\|  \\
&\geq 
\l\|  \begin{pmatrix}
0 & \hat T_{12} - \sum_{j=1}^n \mb E Y_j \\
\hat T^\ast_{12} - \sum_{j=1}^n \mb E Y^\ast_j & 0
\end{pmatrix} \r\| = 
\l\|  \hat T_{12} - \sum_{j=1}^n \mb E Y_j  \r\|,
\end{align*}
and the first inequality follows. 
To obtain the second inequality, it is enough to use Theorem \ref{th:intdim} instead of Theorem \ref{th:main} and note that 
\[
\tr \l(\sum_{j=1}^n \mb E\m H(Y_j)^2\r) = \tr \l(\sum_{j=1}^n \mb E Y_j Y_j^\ast\r) + \tr \l(\sum_{j=1}^n \mb E Y_j^\ast Y_j\r) = 
2\tr \l(\sum_{j=1}^n \mb E Y_j^\ast Y_j\r)
\]
since for any $1\leq j\leq n$, $\tr \l(\mb E Y_j Y_j^\ast\r) =\mb E \tr(Y_j Y_j^\ast)=\mb E \tr(Y_j^\ast Y_j)$.
\end{proof}

In a particular case when $Y\in \mb R^d$ is a random vector such that $\mb EYY^T=\Sigma$ and $Y_1,\ldots,Y_n$ are its i.i.d. copies, 
$\max\l( \l\|  \sum_{j=1}^n \mb E Y_j Y_j^\ast \r\|, \l\|  \sum_{j=1}^n \mb E Y_j^\ast Y_j \r\| \r)=n\,\tr \Sigma$ and 
$\tr \l( \sum_{j=1}^n \mb EY_j^\ast Y_j \r)=n\, \tr\Sigma$, 
hence $\bar d=2$ and the estimator $\hat T_{12}$ admits the following bound: 
if we replace $t$ by $\sqrt{s}\sqrt{\tr\Sigma}$ and set $\theta=\sqrt{\frac{s}{n}}\frac{1}{\sqrt{\tr \Sigma}}$ in the second bound of Corollary \ref{cor:rectangular}, then
\[
\Pr\l( \l\| \frac{\hat T_{12}}{n} - \mb EY \r\|_2\geq \sqrt{\tr \Sigma}\sqrt{\frac{s}{n}}\r)\leq 4\l(1+1/s\r) e^{-s/2}.
\]

\subsection{Bounds under weaker moment assumptions}

In this section, we discuss the mean estimation problem under weaker moment conditions. 
Namely, assume that $Y_1,\ldots,Y_n$ are independent self-adjoint random matrices such that $\| \mb E|Y_j|^\alpha\|<\infty$ for some $\alpha\in (1,2]$ and all $1\leq j\leq n$. 
Let $\psi_\alpha$ satisfy 
\[
-\log(1-x+c_\alpha |x|^\alpha) \leq \psi_\alpha(x) \leq \log(1+x+c_\alpha |x|^\alpha)
\]
for all $x\in \mb R$, where $c_\alpha = \frac{\alpha-1}{\alpha}\vee \sqrt{\frac{2-\alpha}{\alpha}}$. 
The fact that such $\psi_\alpha$ exists follows from Lemma \ref{lemma:psi_alpha} in the appendix. 
For example, one can take $\psi_\alpha(x)=\log(1+x+c_\alpha |x|^\alpha)$. 
The following result holds:
\begin{theorem}
\label{th:psi_alpha}
Assume that $v^\alpha_{n}\geq \l\|  \sum_{j=1}^n \mb E|Y_j|^\alpha \r\|$. 
Then for any positive $t$ and $\theta$,
\[
\Pr\l( \l\| \sum_{j=1}^n \l(\frac{1}{\theta}\psi_\alpha(\theta Y_j) - \mb E Y_j \r) \r\| \geq t \r)\leq 2d\exp\l( -\theta t + c_\alpha \theta^\alpha v^\alpha_{n}\r).
\]
\end{theorem}
\begin{proof}
The argument repeats the steps of Lemma \ref{lemma:main} and Theorem \ref{th:main}, the only difference being that application of Fact \ref{fact:04} is replaced by Lemma \ref{lemma:psi_alpha}.
\end{proof}
\begin{remark}
In the special case when $Y_1,\ldots,Y_n$ are i.i.d. copies of $Y$ with $v=\| \mb E|Y|^\alpha \|^{1/\alpha}$, setting 
$t=v n^{1/\alpha} s^{\frac{\alpha-1}{\alpha}}$ and 
$\theta=\l(\frac{1}{\alpha c_\alpha} \r)^{1/(\alpha-1)}\l(\frac{s}{n}\r)^{1/\alpha}\frac{1}{v}$ gives the inequality 
\[
\Pr\l( \l\| \frac{1}{n\theta}\sum_{j=1}^n \psi_\alpha(\theta Y_j) - \mb E Y  \r\| \geq  v \l(\frac{s}{n}\r)^{\frac{\alpha-1}{\alpha}} \r)\leq 
2d\exp\l( -\frac{\alpha-1}{\alpha}\l(\frac{1}{\alpha c_\alpha}\r)^{1/(\alpha-1)}s\r).
\]
Note that for $\alpha=2$, we recover (\ref{th:iid}).
\end{remark}

Before we proceed with discussion or further improvements and adaptation issues, let us demonstrate applications of developed techniques to popular problems in statistics and highlight the advantages over existing results.

\section{Examples}
\label{section:examples}

We present two examples which highlight the potential improvements obtained via our technique in popular scenarios: estimation of the covariance matrix in Frobenius and operator norms, and low-rank matrix completion problem. 

\subsection{Estimation of the covariance matrix in operator norm}

Let $Z\in \mb R^d$ be a random vector with 
$\mb EZ=\mu$, $\mb E \|Z-\mu\|_2^4<\infty$, $\Sigma=\mb E\l[ (Z-\mu)(Z-\mu)^T \r]$, and let 
$Z_1,\ldots, Z_{2n}$ be i.i.d. copies of $Z$. 
Let us first assume that $\mu=0$, and define 
\begin{align*}
&
\widetilde \Sigma_{2n}(\theta) = \frac{1}{2n\theta}\sum_{j=1}^{2n} \psi\l(  \theta Z_j Z_j^T  \r), 
\end{align*}
where $\psi(\cdot)$ satisfies (\ref{eq:psi}). 
Let $\sigma^2\geq\l\|\mb E\|Z\|_2^2 Z Z^T \r\|$ and $\tilde\theta = \sqrt{\frac{t}{n}}\frac{1}{\sigma}$. 
It is straightforward to deduce from Theorem \ref{th:main} that with probability $\geq 1- 2d e^{-t}$,
\[
\l\| \widetilde\Sigma_{2n}(\tilde\theta) - \Sigma\r\|\leq \sigma \sqrt{\frac{t}{n}}.
\]
\begin{remark}
\label{remark:sigma}
\begin{asparaenum}
\item 
Note that for any matrix $X=\lambda U U^T$ of rank $1$ (where $\|U\|_2=1$), 
\[
\psi(X)=\psi(\lambda)U U^T \text{ (since $\psi(0)=0$)},
\] 
hence 
$\tilde \Sigma_{2n}(\tilde\theta) = \frac{1}{2n \tilde\theta}\sum_{j=1}^{2n} \psi(\tilde\theta \|Z_j\|_2^2)\frac{Z_j Z_j^T}{\|Z_j\|_2^2}$. 
In particular, this expression is easy to evaluate numerically; in general, computation of the estimator \eqref{eq:T^0} requires $n$ singular value decompositions. 
\item 
Parameter $\sigma$ is closely related to the \textit{effective rank} defined as 
$\mathrm{r}(\Sigma)=\frac{\tr(\Sigma)}{\|\Sigma\|}$ \cite{vershynin2010introduction}; clearly, it always true that $\mathrm{r}(\Sigma)\leq d$.  
The quantity $\sqrt{\mathrm{r}(\Sigma)} \|\Sigma\|$ has been shown to control the expected error of the sample covariance estimator in the Gaussian setting \cite{koltchinskii2017concentration}. 
Under the additional assumption that the kurtosis of the linear forms $\langle Z,v \rangle, \ v\ne 0$, is uniformly bounded by $K$, 
it is possible to show that (see Lemma 2.3 in \cite{S.-Minsker:2017aa}) that 
$\sigma^2\leq K \,\mathrm{r}(\Sigma) \, \|\Sigma\|^2$.  
On the other hand, fluctuations of the error around its expected value in the Gaussian case  \cite{koltchinskii2017concentration} are controlled by the ``weak variance'' 
$\sup_{v\in\mb R^d:\|v\|_2=1}\mb E^{1/2}\dotp{Z}{v}^4\leq \sqrt{K} \|\Sigma\|$, while in our bounds fluctuations are controlled by the ``strong variance'' $\sigma^2$; this fact leaves room for improvement in our construction and proof techniques. 
\end{asparaenum}
\end{remark}

Of course, the initial assumption that $\mu$ is known is often unrealistic, hence we modify the estimator as follows. 
Given $\theta>0$, set 
\begin{align*}
&Y_j = \frac{1}{2}\l( Z_{2j-1} - Z_{2j} \r)\l( Z_{2j-1} - Z_{2j} \r)^T,
\\
&\widehat \Sigma_{2n}(\theta) = \frac{1}{n\theta}\sum_{j=1}^n \psi(\theta Y_j). 
\end{align*}
Let
$
\hat\sigma^2\geq \frac{1}{2} \l\|  \mb E\l(  (Z-\mu)(Z-\mu)^T \r)^2 + \tr(\Sigma)\Sigma +2\Sigma^2 \r\|, 
$ 
and $\hat\theta = \sqrt{\frac{t}{n}}\frac{1}{\hat\sigma}$. 
Our covariance estimator is then defined as 
$\widehat \Sigma_{2n}:=\widehat \Sigma_{2n}(\hat\theta).$  
The following result can be deduced from Theorem \ref{th:main}:
\begin{corollary}
\label{corollary:cov}
With probability $\geq 1-2d\, e^{-t}$,
\[
\l\| \widehat\Sigma_{2n} - \Sigma  \r\| \leq \sqrt{2}\hat\sigma\sqrt{\frac{t}{n}}.
\]
\end{corollary}
\noindent Before presenting the proof, let us make several additional remarks.
\begin{remark}
\begin{asparaenum}
\item 
It is not hard to show that (see Corollary \ref{corollary:FKG-bound}) that  
\[
\l\| \mb E\l(  (Z-\mu)(Z-\mu)^T \r)^2 \r\| \geq \tr(\Sigma) \l\| \Sigma \r\|, 
\]
hence it is enough to choose
$\hat\sigma^2\geq \| \Sigma \|^2 + \sigma_0^2 = \| \Sigma \|^2 + \l\| \mb E\l(  (Z-\mu)(Z-\mu)^T \r)^2 \r\|$. 
In view of remark \ref{remark:sigma}, this expression can be further simplified under the bounded kurtosis assumption, 
and one can choose $\hat\sigma^2\geq \|\Sigma\|^2 \l( 1+K \mathrm{r}(\Sigma) \r)$, where $K$ is the uniform bound on the kurtosis of the coordinates of $Z$, and $\mathrm{r}(\Sigma)$ is the effective rank.
\item
Construction of $\widehat \Sigma_{2n}(\theta)$ essentially halves the effective sample size. 
While the loss of a constant factor can be deemed insignificant in non-asymptotic theoretical bounds, it is undesirable in applications. 
A more natural version of the estimator based on a sample of size $2n$ is the U-statistic 
\[
\bar \Sigma_{2n}(\theta) = \frac{1}{{2n\choose 2}}\sum_{1\leq i<j\leq 2n} \frac{1}{\theta}
\psi\l(  \frac{\theta}{2}(Z_i-Z_j)(Z_i-Z_j)^T \r).
\] 
Another possibility to avoid ``halving'' the sample size is to center the data using a robust estimator of location, such as the spatial median or the median-of-means estimator \citep{joly2017estimation,lugosi2017sub,minsker2015geometric}. 
Analysis of the estimators of these types is not covered in the present paper, and requires a slightly different set of technical tools to deal with dependent summands; see \cite{S.-Minsker:2017aa} for results in this direction.
\end{asparaenum}
\end{remark}
\begin{proof}[Proof of Corollary \ref{corollary:cov}]
Note that for all $j=1,\ldots,n$, $\mb E Y_j=\Sigma$. 
Since $Y_1,\ldots,Y_n$ are i.i.d. random matrices, Theorem \ref{th:main} applies (see remark \ref{remark:iid}), giving that
\[
\Pr\l( \l\| \hat\Sigma(\hat\theta) -\Sigma  \r\| \geq \hat\sigma\sqrt{\frac{2t}{n}} \r)\leq 2d e^{-t},
\]
where $\hat\sigma^2\geq \l\| \mb E Y_1^2 \r\|$. 
It is easy to check that 
\[
\l\| \mb E Y_1^2 \r\|=\frac{1}{2} \l\|  \mb E\l(  (Z-\mu)(Z-\mu)^T \r)^2 + \tr(\Sigma)\Sigma +2\Sigma^2 \r\|,
\]
and result follows. 
\end{proof}

\subsection{Estimation of the covariance matrix in Frobenius norm}
\label{section:frobenius}
Next, we present an estimator which achieves strong deviation guarantees in the Frobenius norm. 
Estimation of the covariance matrix with respect to this norm has been previously investigated in the literature, for instance, see \cite{lam2009sparsistency}, \cite{cai2010optimal} and references therein; Frobenius norm is a natural choice when one wants to understand the effect of the rank of an unknown covariance matrix on the estimation error \cite{lounici2014high}. 
Let $\hat S_{2n}$ be the sample covariance estimator based on $Z_1,\ldots,Z_{2n}$:
\[
\hat S_{2n} = \frac{1}{{2n\choose 2}}\sum_{1\leq i<j\leq 2n}\frac{(Z_i-Z_j)(Z_i-Z_j)^T}{2}.
\] 
The following ``soft thresholding'' estimator has been studied in \cite{lounici2014high}; here, $\tau>0$ is a fixed threshold parameter: 
\begin{align}
&
\hat S_{2n}^\tau=\argmin_{A\in \mb R^{d\times d}}\l[  \l\| A - \hat S_{2n} \r\|^2_{\mathrm{F}} +\tau \l\| A \r\|_1\r].
\end{align}
We propose to replace the sample covariance $\hat S_{2n}$ by $\widehat \Sigma_{2n}$, and consider 
\begin{align}
&
\widehat \Sigma_{2n}^\tau=\argmin_{A\in \mb R^{d\times d}}\l[  \l\| A - \widehat \Sigma_{2n} \r\|^2_{\mathrm{F}} + \tau \l\| A \r\|_1\r].
\end{align}
It is not hard to see (e.g., see the proof of Theorem 1 in \cite{lounici2014high}) that $\widehat \Sigma_{2n}^\tau$ can be written explicitly as 
\[
\widehat \Sigma_{2n}^\tau = \sum_{j=1}^d \max\l(\lambda_j\l(\widehat \Sigma_{2n}\r) -\tau/2, 0\r) v_j(\widehat \Sigma_{2n}) v_j(\widehat \Sigma_{2n})^T,
\]
where $\lambda_j(\widehat \Sigma_{2n})$ and $v_j(\widehat \Sigma_{2n})$ are the eigenvalues and corresponding eigenvectors of $\widehat \Sigma_{2n}$. 
The following result holds: 
\begin{theorem}
\label{th:covariance}
For any 
\[
\tau \geq 4\hat \sigma\sqrt{\frac{t+\log(2d)}{2n}}
\]
\begin{align}
&\label{eq:ex70}
\l\| \widehat \Sigma^\tau_{2n} - \Sigma \r\|_{\mathrm{F}}^2\leq \inf_{A\in \mb R^{d\times d}} \l[  \l\| A - \Sigma \r\|_{\mathrm{F}}^2 + \frac{(1+\sqrt{2})^2}{8}\tau^2\mathrm{rank}(A)  \r].
\end{align}
with probability $\geq 1-e^{-t}$.
\end{theorem}
Result stated above mimics the (almost) optimal rates obtained in \cite{lounici2014high} (in the situation when no data is missing) under significantly weaker assumptions on the underlying distribution. 
\begin{proof}[Proof of Theorem \ref{th:covariance}]
The proof is based on the following lemma:
\begin{lemma}
Inequality (\ref{eq:ex70}) holds on the event $\m E=\l\{ \tau\geq 2\l\| \widehat \Sigma_{2n} - \Sigma \r\| \r\}$. 
\end{lemma}
\noindent To verify this statement, it is enough to repeat the steps of the proof of Theorem 1 in \cite{lounici2014high}, replacing each occurrence of the sample covariance $\hat S_{2n}$ by its robust counterpart $\widehat \Sigma^\tau_{2n}$. \\
Result then follows from corollary \ref{corollary:cov} that $\Pr(\m E)\geq 1-e^{-t}$ whenever $\tau \geq 4\hat\sigma\sqrt{\frac{t+\log(2d)}{2n}}$.
\end{proof}

\subsection{Matrix completion}

Let $A_0\in \mb R^{d_1\times d_2}$ be an unknown matrix, and assume that we observe a random subset of its entries contaminated by noise. 
The goal is to estimate $A_0$ from a small number of such noisy measurements under an additional assumption that $A_0$ is likely to be of low rank (or can be well approximated by a low rank matrix). 
More specifically, let 
\[
\m X=\l\{  e_j(d_1)e_k^T(d_2), \ 1\leq j\leq d_1,  \  1\leq k\leq d_2\r\},
\]
where $e_j(d_1)$ and $e_k(d_2)$ are the elements of the canonical bases of $\mb R^{d_1}$ and $\mb R^{d_2}$ respectively. 
Let $X$ have uniform distribution $\Pi:=\mathrm{Unif}(\m X)$ on $\m X$, and assume that the noisy linear measurement $Y$ has the form 
\[
Y=\tr(X^T A_0)+\xi, 
\]
where $\mb E(\xi|X)=0$. 
Finally, assume that $(X_1,Y_1),\ldots,(X_n,Y_n)$ are i.i.d. copies of $(X,Y)$. 

It is easy to check that $\mb E(YX)=\frac{1}{d_1 d_2}A_0$, hence the natural unbiased estimator of $A_0$ is 
\[
\widehat A=\frac{d_1 d_2}{n}\sum_{j=1}^n Y_j X_j.
\]
To incorporate the structural (low-rank) assumption on $A_0$, the following estimator has been considered in the literature: let $\tau>0$, and define
\begin{align*}
\widehat A^\tau&=\argmin_{A\in \mb R^{d_1\times d_2}}\l[  \frac{1}{d_1 d_2}\|  A - \widehat A \|_{\mathrm{F}}^2 +\tau \|A\|_1\r] \\
&
=\argmin_{A\in \mb R^{d_1\times d_2}}\l[ \frac{1}{d_1 d_2} \|A\|_{\mathrm{F}}^2 - \dotp{\frac{2}{n}\sum_{j=1}^n Y_j X_j}{A} + \tau \|A\|_1  \r].
\end{align*}
Note that one can use the symmetric version $\widehat A_s\in \mb R^{(d_1+d_2)\times (d_1+d_2)}$ of $\widehat A$ instead, defined as
\[
\widehat A_s=\frac{d_1 d_2}{n}\sum_{j=1}^n Y_j \m H(X_j),
\]
so that $\mb E\widehat A_s=\m H(A_0)$, and consider the equivalent convex minimization problem
\begin{align*}
\widehat A^\tau&=\argmin_{A\in \mb R^{d_1\times d_2}}\l[  \frac{1}{d_1 d_2}\|  \m H(A) - \m H(\widehat A_s) \|_{\mathrm{F}}^2 +2 \tau \|A\|_1\r] \\
&
=\argmin_{A\in \mb R^{d_1\times d_2}}\l[ \frac{1}{d_1 d_2} \|\m H(A)\|_{\mathrm{F}}^2 - \dotp{\frac{2}{n}\sum_{j=1}^n Y_j \m H(X_j)}{\m H(A)} +2\tau \|A\|_1  \r].
\end{align*}
However, strong theoretical guarantees for this estimator exist only when the ``noise term'' $\xi_j$ is either bounded with probability 1, or has sub-exponential tails. 
We propose to replace $\widehat A_s$ with a robust estimator 
\begin{align*}
&
\widehat R=\frac{d_1 d_2}{n\theta}\sum_{j=1}^n \psi\l(\theta Y_j \m H(X_j)\r), 
\end{align*}
where $\psi(\cdot)$ satisfies (\ref{eq:psi}) 
and 
\[
\theta:=\theta(t,n,A_0)=\frac{1}{\|A_0\|_{\max}\vee\sqrt{\var(\xi)}}\sqrt{\frac{(t+\log(2(d_1+d_2)))(d_1\wedge d_2)}{n}}.
\] 
The reasoning behind this choice of $\theta$ is explained below. 
Consider 
\begin{align*}
&
\widehat R^\tau=\argmin_{A\in \mb R^{d_1\times d_2}}
\l[ \frac{1}{d_1 d_2} \| \m H(A) \|_{\mathrm{F}}^2 - \dotp{\frac{2}{d_1 d_2}\widehat R}{\m H(A)} +2\tau \|A\|_1  \r].
\end{align*}
Finally, set 
\[
M = \widehat R - \mb E \l(Y\m H(X)\r).
\]
The following result holds:

\begin{theorem}
\label{th:completion}
Assume that $\xi_j$ is independent of $X_j, \ j=1,\ldots,n$, and that $\var(\xi)<\infty$. 
For any 
\[
\tau\geq 4\l(\|A_0\|_{\max}\vee\sqrt{\var(\xi)}\r)\sqrt{\frac{t+\log (2(d_1+d_2))}{n(d_1\wedge d_2)}},
\]
\[
\frac{1}{d_1 d_2}\l\|\widehat R^\tau - A_0 \r\|^2_{\textrm{F}}\leq 
\inf_{A\in \mb R^{d_1\times d_2} } \l[ \frac{1}{d_1 d_2} \l\| A-A_0 \r\|^2_{\mathrm{F}} +\l(\frac{1+\sqrt 2}{2} \r)^2 d_1 d_2 \tau^2 \rank(A) \r].
\]
with probability $\geq 1-e^{-t}$.
\end{theorem}
\noindent Note that we only assume that $\var(\xi)<\infty,$ while in \cite{fan2016robust}, a similar result is obtained under a slightly stronger assumption requiring that $\mb E | \xi |^{2+\eps}<\infty$ for some $\eps>0$.
\begin{proof}
Define $\mb A\subseteq \mb R^{(d_1+d_2)\times (d_1+d_2)}$ to be the image of $\mb R^{d_1\times d_2}$ under $\m H(\cdot)$:
\[
\mb A=\l\{ B\in \mb R^{(d_1+d_2)\times (d_1+d_2)}: \ B=\m H(A) \text{ for some } A\in \mb R^{d_1\times d_2} \r\}.
\]
We begin with the following inequality:
\begin{lemma}
\label{lemma:completion}
Assume that $\tau\geq 2\|M\|$. Then
\begin{align*}
\frac{1}{d_1 d_2}\l\| \m H(\widehat R^\tau)  - \m H(A_0) \r\|^2_{\textrm{F}}\leq 
\inf_{B\in \mb A} \l[ \frac{1}{d_1 d_2} \l\| B-\m H(A_0) \r\|^2_{\mathrm{F}} +\l(\frac{1+\sqrt 2}{2} \r)^2 d_1 d_2 \tau^2 \rank(B) \r].
\end{align*}
\end{lemma}
\begin{proof}
By the definition of $\widehat R^{\tau}$, we see that 
\begin{align*}
&
\m H(\widehat R^\tau)=\argmin_{B\in \mb A}
\l[ \frac{1}{d_1 d_2} \| B \|_{\mathrm{F}}^2 - \dotp{\frac{2}{d_1 d_2}\widehat R}{B} +\tau \|B\|_1  \r].
\end{align*}
If we replace $\frac{1}{d_1 d_2}\widehat R$ by $\frac{1}{d_1 d_2}\widehat A_s=\frac{1}{n}\sum_{j=1}^n Y_j \m H(X_j)$, the result follows from Theorem 1 in \cite{koltchinskii2011nuclear} immediately. 
To obtain the current statement, it is enough to repeat the argument of Theorem 1 in \cite{koltchinskii2011nuclear}, replacing each occurrence of the matrix 
$\frac{1}{d_1 d_2}\widehat A_s$ by $\frac{1}{d_1 d_2}\widehat R$. 
\end{proof}
\noindent 
To complete the proof, we will estimate each side of the inequality of Lemma \ref{lemma:completion}. 
First, it is obvious from the definition of the Frobenius norm that 
\begin{align}
\label{eq:ex10}
\frac{1}{d_1 d_2}\l\| \m H(\widehat R^\tau) - \m H(A_0) \r\|^2_{\textrm{F}} = \frac{2}{d_1 d_2}\l\| \widehat R^\tau - A_0 \r\|^2_{\textrm{F}}.
\end{align}
\noindent Next, since $\rank(\m H(A))=2\rank(A)$,
\begin{align}
\label{eq:ex20}
&\nonumber
\inf_{B\in \mb A} \l[ \frac{1}{d_1 d_2} \l\| B - \m H(A_0) \r\|^2_{\mathrm{F}} +\l(\frac{1+\sqrt 2}{2} \r)^2 d_1 d_2 \tau^2 \rank(B) \r] \\
&
=2 \inf_{A\in \mb R^{d_1\times d_2}} \l[ \frac{1}{d_1 d_2} \l\| A-A_0 \r\|^2_{\mathrm{F}} +\l(\frac{1+\sqrt 2}{2} \r)^2 d_1 d_2 \tau^2 \rank(A) \r].
\end{align}
It remains to estimate the probability of the event $\m E=\l\{ \tau\geq 2\|M\| \r\}$.
Let 
\[
\sigma^2:=\max\l( \l\|  \mb E \l[ Y^2 XX^T\r] \r\|, \l\|  \mb E \l[Y^2 X^T X\r] \r\|  \r).
\] 
\begin{lemma}
Assume that $\xi_j$ is independent of $X_j$, $j=1,\ldots,n$. Then
\[
\sigma^2 \leq \l(\var(\xi)\vee\|A_0\|^2_{\max}\r) \frac{2}{d_1\wedge d_2}.
\]
\end{lemma}
\begin{proof}
Note that $\mb E \l[Y^2 XX^T\r] = \mb E \l[\xi^2 XX^T\r]  + \mb E \l[ \l(\tr(X^T A_0)\r)^2 X X^T \r]$. 
Moreover, $|\tr(X^T A_0)|\leq \max_{i,j}\l|  (A_0)_{i,j} \r|=\|A_0\|_{\max}$, and $\l\| \mb EX X^T = \frac{1}{d_1}\r\|$, hence
\[
\l\|  \mb E \l[Y^2 XX^T\r]  \r\|\leq \var(\xi)\frac{1}{d_1} + \|A_0\|^2_{\max}\frac{1}{d_1}.
\]
Similarly, 
\[
\l\|  \mb E \l[Y^2 X^T X\r]  \r\|\leq \var(\xi)\frac{1}{d_2} + \|A_0\|^2_{\max}\frac{1}{d_2}.
\]
\end{proof}
\noindent Applying Theorem \ref{th:main} (see remark \ref{remark:iid}) with 
\begin{align*}
\theta&=
\sqrt{\frac{2(t + \log(2(d_1+d_2)))}{n}}\frac{1}{\l(  \l(\var(\xi)\vee\|A_0\|^2_{\max}\r) \frac{2}{d_1\wedge d_2}  \r)^{1/2}}\\
&
= \frac{1}{\|A_0\|_{\max}\vee\sqrt{\var(\xi)}}\sqrt{\frac{(t+\log (2(d_1+d_2)))(d_1\wedge d_2)}{n}},
\end{align*} 
we see that 
\[
\|M\|\leq 2\l(\|A_0\|_{\max}\vee\sqrt{\var(\xi)}\r)\sqrt{\frac{t+\log (2(d_1+d_2))}{n(d_1\wedge d_2)}}
\]
with probability $\geq 1-e^{-t}$. 
Final result now follows from the combination of this inequality with (\ref{eq:ex10}), (\ref{eq:ex20}) and Lemma \ref{lemma:completion}.

\end{proof}

\section{Optimal choice of $\theta$ and adaptation to the unknown second moment}
\label{section:variance}

To make results of Theorem \ref{th:main} useful, one has to set the value for the parameter $\theta$ which in turn depends on the (usually unknown) norm $\sigma_n^2=\l\|\sum_{j=1}^n \mb EY_j^2\r\|$. 
To address this problem, we develop a simple adaptive solution based on Lepski's method.

Lepski's method \cite{lepskii1992asymptotically} is a powerful general technique that allows to adapt to the unknown structure of the problem - for example, bandwidth selection in nonparametric estimation, or unknown second moment in our case. 
Let $Y_1,\ldots,Y_n\in \mb C^{d\times d}$ be independent self-adjoint random matrices with 
$\sigma_n^2=\l\| \sum_{j=1}^n\mb EY_j^2 \r\|$, and assume that $\sigma_{\mn}, \ \sigma_{\mx}$ are such that 
\[
\sigma_{\mn}\leq \frac{\sigma_n}{\sqrt n} \leq \sigma_{\mx}.
\]
Parameters $\sigma_{\mn}$ and $\sigma_{\mx}$ are ``crude'' preliminary bounds that can differ from $\sigma_n/\sqrt{n}$ by several orders of magnitude. 
Let $\sigma_j = \sigma_{\mn} 2^j$ and
\[
\m J=\l\{ j\in \mb Z: \  \sigma_{\mn} \leq \sigma_j  < 2\sigma_{\mx} \r\}
\]
be a set of cardinality $|\m J|\leq 1+\log_2(\sigma_{\mx}/\sigma_{\mn})$, 
and for each $j\in \m J$ set $\theta_j=\theta(j,t)=\sqrt{\frac{2t}{n}}\frac{1}{\sigma_j}$. 
Define
\[
T_{n,j}=\frac{1}{n\theta_j}\sum_{i=1}^n \psi(\theta_j Y_i),
\]
where $\psi(\cdot)$ satisfies (\ref{eq:psi}). 
Finally, set 
\begin{align}
\label{eq:lepski}
j_\ast:=\min\l\{ j\in \m J: \forall k>j \text{ s.t. } k\in \m J,\ \l\|  T_{n,k} - T_{n,j} \r\|\leq 2 \sigma_{k} \sqrt{\frac{2t}{n}}  \r\}
\end{align}
and $T_n^\ast:=T_{n,j_\ast}$. 

Next result shows that adaptation is possible at the cost of an additional multiplicative constant factor $6$ in the deviation bound. 
\begin{theorem}
\label{th:lepski}
The following inequality holds for any $t>0$:
\[
\Pr\l( \l\| T_n^\ast - \mb EY \r\| \geq 6(\sigma_n/\sqrt{n}) \sqrt{\frac{2t}{n}} \r)\leq  2d \log_2\l(\frac{2\sigma_{\mx}}{\sigma_{\mn}}\r) e^{-t}.
\]
\end{theorem}
\begin{proof}
Let $\bar j=\min\l\{  j\in \m J: \ \sigma_j \geq \frac{\sigma_n}{\sqrt n}\r\}$ (hence $\sigma_{\bar j}\leq 2\frac{\sigma_n}{\sqrt n}$). 
First, we will show that $j_\ast \leq \bar j$ with high probability. 
Indeed, 
\begin{align*}
\Pr\l( j_\ast > \r. & \l. \bar j\r)\leq \Pr\l( \bigcup_{k\in \m J: k>\bar j} \l\{  \l\|  T_{n,k} - T_{n,\bar j} \r\| > 2\sigma_k \sqrt{\frac{2t}{n}} \r\} \r)\\
& 
\leq \Pr\l(  \l\|  T_{n,\bar j} - \mb EY \r\| > \sigma_{\bar j} \sqrt{\frac{2t}{n}} \r) + 
\sum_{k\in \m J: \ k>\bar j}\Pr\l( \l\|  T_{n,k} - \mb EY \r\| > \sigma_k \sqrt{\frac{2t}{n}}  \r)   \\
&
\leq 2de^{-t} + 2d \log_2\l(\frac{\sigma_{\mx}}{\sigma_{\mn}}\r) e^{-t},
\end{align*}
where we used Theorem \ref{th:main} to bound each of the probabilities in the sum. 
The display above implies that the event 
\[
\m B = \bigcap_{k\in \m J: k\geq \bar j} 
\l\{ \l\|  T_{n,k} - \mb EY \r\|\leq \sigma_k\sqrt{\frac{2t}{n}}  \r\} 
\]
of probability $\geq 1-2d \log_2\l(\frac{2\sigma_{\mx}}{\sigma_{\mn}}\r) e^{-t}$ is contained in $\m E=\l\{  j_\ast\leq \bar j \r\}$. 
Hence, on $\m B$ we have that 
\begin{align*}
\l\| T_n^\ast - \mb EY \r\|&
\leq \| T_n^\ast - T_{n,\bar j} \| + \| T_{n,\bar j} - \mb EY \| \leq 
2 \sigma_{\bar j}\sqrt{\frac{2t}{n}} + \sigma_{\bar j}\sqrt{\frac{2t}{n}} \\
&\leq 4\frac{\sigma_n}{\sqrt n} \sqrt{\frac{2t}{n}} + 2\frac{\sigma_n}{\sqrt n}\sqrt{\frac{2t}{n}} = 6\frac{\sigma_n}{\sqrt n} \sqrt{\frac{2t}{n}},
\end{align*}
and result follows. 
\end{proof}
\begin{remark}
It follows from the proof that constant factor 6 in Theorem \ref{th:lepski} can be reduced to $3+\eps$ for any $\eps>0$ by considering the ``finer grid'', that is, replacing $\m J$ by 
$\l\{ j\in \mb Z: \sigma_{\mn} \leq \kappa^j\sigma_{\mn} < \kappa\sigma_{\mx}\r\}$ for some $1<\kappa<2$, at the cost of replacing 
$\log_2\l(\frac{2\sigma_{\mx}}{\sigma_{\mn}}\r)$ by $\log_2\l(\frac{\kappa\sigma_{\mx}}{\sigma_{\mn}}\r)/\log_2 \kappa$. 
\end{remark}

\section{From bounds depending on $\|\mb EY^2\|$ to bounds depending on $\|\mb E(Y-\mb EY)^2\|$}
\label{section:shift}

Assume that $Y_1,\ldots,Y_n$ are i.i.d. copies of $Y\in \mb C^{d\times d}$. 
In this section, we build upon previously established bounds to provide performance guarantees for the estimator defined via \eqref{eq:matrix-catoni}, \eqref{eq:matrix-catoni2}. 
To this end, we study a version of the steepest descent scheme for the problem \eqref{eq:matrix-catoni} initialized at the point $\wh T^{(0)}_\theta$, namely, $\hat T_0:= \wh T^{(0)}_{\theta_0}$ and
\begin{align*}
& \hat T_k = \hat T_{k-1} + \frac{1}{n\theta_k}\sum_{j=1}^n \psi \l( \theta_k (Y_j - \hat T_{k-1}) \r), \ k\geq 1
\end{align*}
for an appropriate choice of $\theta_k, \ k\geq 0$. 
Note that for any non-random self-adjoint matrix $S$ and $\theta_S=\sqrt{\frac{s}{n}}\frac{1}{\|\mb E(Y-S)^2\|^{1/2}}$, Theorem \ref{th:main} implies that 
\[
\Pr\l( \l\| T_n(S) - \mb EY \r\| \geq \|\mb E(Y-S)^2\|^{1/2} \sqrt{\frac{s}{n}} \r)
\leq 2d \exp\l( -s/2 \r),
\]
where $T_n(S)=S+\frac{1}{n\theta_S}\sum_{j=1}^n \psi\l( \theta_S (Y_j-S)  \r)$. 
Hence, if we use random $S$ which is ``not too far'' from $\mb EY$ with high probability, we expect that the deviation guarantees will still hold with the ``variance parameter'' close to $\l\| \mb E(Y - \mb EY)^2 \r\|$.

Everywhere in this section, we will assume that one has access to some known (possibly very crude) bounds for $\sigma^2=\| \mb E Y^2 \|$ and 
$\sigma_0^2=\| \mb E(Y-\mb EY)^2 \|$:
\begin{assumption}
\label{assump:variance}
Let $\sigma_{\mn},\sigma_{0,\mn}$ and $\sigma_{\mx},\sigma_{0,\mx}$ be known constants such that 
\[
\sigma_{\mn}\leq \sigma \leq \sigma_{\mx} \text{ and } \sigma_{0,\mn}\leq \sigma_0 \leq \sigma_{0,\mx}.
\] 
\end{assumption}

\subsection{Two-step estimation based on sample splitting}

We will first discuss the simplest (but not the most efficient) approach based on splitting the sample $Y_1,\ldots, Y_n$ into two disjoint subsets $G_1$ and $G_2$ of cardinality $\geq \lfloor n/2\rfloor$ each, and performing one step of the steepest descent.  
The main advantage of this approach is the fact that it requires very mild assumptions. 
The idea is to apply Lepski's method (as discussed in section \ref{section:variance}) twice: on the first step, we obtain an estimator $\hat T_0$ based on subsample $G_1$, and on the second step we apply Lepski's method again to the subsample 
$\l\{ Y_j - \hat T_0: \ 1\leq j\leq n,\ Y_j\in G_2\r\}$.  

\noindent Here is the more detailed description: set $\sigma_j = 2^j\sigma_{\mn}$,
\[
\m J_1=\l\{ j\in \mb Z: \  \sigma_{\mn} \leq \sigma_j  < 2\sigma_{\mx} \r\}
\] 
and $\sigma_{0,j} = 2^j\sigma_{0,\mn}$,
\[
\m J_2=\l\{ j\in \mb Z: \  \sigma_{0,\mn} \leq  \sigma_{0,j}  < 2\l( \sigma_{0,\mx} + 12\sigma_{\mx}\sqrt{\frac{t}{n}}\r) \r\},
\] 
and let $\hat T_0$ be the ``Lepski-type'' adaptive estimator based on the subsample $G_1$ defined as
\[
\hat T_0=T_{|G_1|,j_1^\ast}(0;G_1),
\] 
where 
\[
T_{|G_1|,j}(S; G_1)=\frac{1}{|G_1|\theta_j}\sum_{i=1}^{|G_1|} \psi(\theta_j (Y_i - S)),
\]
$\theta_j=\sqrt{\frac{2t}{n/2}}\frac{1}{\sigma_j}$, $\psi(\cdot)$ satisfies (\ref{eq:psi}) and 
\begin{align*}
j_1^\ast:=\min \Bigg\{ j\in \m J_1: & \forall k\in \m J_1 \text{ s.t. } k>j, \\
&
 \l\|  T_{|G_1|,k}(0;G_1) - T_{|G_1|,j}(0;G_1) \r\|\leq 2\sigma_k \sqrt{\frac{2t}{|G_1|}}  \Bigg\}
\end{align*}
$\hat T_1$ is then defined as follows: 
\[
\hat T_1=\hat T_0 + T_{|G_2|,j_2^\ast}(\hat T_0;G_2),
\]
where
\[
T_{|G_2|,j}(S;G_2)=\frac{1}{|G_2|\theta_{0,j}}\sum_{i=|G_1|+1}^{n} \psi(\theta_{0,j} (Y_i - S)), \quad 
\theta_{0,j}=\sqrt{\frac{2t}{n/2}}\frac{1}{\sigma_{0,j}}
\]
and
\begin{align*}
j_2^\ast:=\min \Bigg\{ j\in \m J_2: & \forall k\in \m J_2  \text{ s.t. } k>j, \\ 
& 
\l\|  T_{|G_2|,k}(\hat T_0;G_2) - T_{|G_2|,j}(\hat T_0;G_2) \r\|\leq 2\sigma_{0,k} \sqrt{\frac{2t}{|G_2|}}  \Bigg\}.
\end{align*}
\begin{theorem}
\label{th:adaptive}
With probability at least
\[
1- 2d\l(2+ \log_2\l( \frac{\sigma_{\mx}}{\sigma_{\mn}}\r) + \log_2\l(\frac{\sigma_{0,\mx}+12\sigma_{\mx}\sqrt{t/n}}{\sigma_{0,\mn}}\r)  \r)e^{-t},
\]
the following inequality holds:
\begin{align*}
\l\| \hat T_1 -\mb EY \r\| \leq 12\l(\sigma_0 + 12\sigma\sqrt{\frac t n} \r)\sqrt{\frac{t}{n}}.
\end{align*}
\end{theorem}
\begin{proof}
See appendix \ref{section:proof-adaptive}.
\end{proof}
The main feature of this result is the variance term $\sigma_0 + 12\sigma\sqrt{\frac t n}$ that can be much smaller compared to $\sigma$ as long as $t\ll n$.

\subsection{Results for the estimator $\wh T_\theta^\ast$ defined via equation \eqref{eq:matrix-catoni2}}
\label{section:chaining}

We will next show how to design an estimator with deviations controlled by ``correct'' variance term without sample splitting (however, subject to the condition that the sample size is sufficiently large). 
In what follows, we will make an additional assumption about the function $\psi$:
\begin{assumption}
\label{ass:2}
Function $\psi(\cdot)$ satisfies (\ref{eq:psi}) and is operator Lipschitz, meaning that 
$\l\| \psi(A) - \psi(B) \r\|\leq L \|A-B\|$ for all self-adjoint $A,B\in \mb C^{d\times d}$, with Lipschitz constant $L$ \textit{independent} of the dimension $d$. 
\end{assumption}
\noindent 
For example, we may take $\psi=\psi_1$ or $\psi=\psi_2$ (see Lemma \ref{lemma:lipschitz} for details). 
As before, let $t>0$ be fixed, set $\sigma_{0,j} = 2^j\sigma_{0,\mn}$,
\[
\m J=\l\{ j\in \mb Z: \  \sigma_{0,\mn} \leq \sigma_{0,j} < 2\sigma_{0,\mx} \r\},
\]
\[
\theta =  \sqrt{\frac{2t}{n}}\frac{1}{\sigma_{\mx}} \text{ and }\theta_j=\sqrt{\frac{2t}{n}}\frac{1}{\sigma_{0,j}} \text{ for } j\in \m J.
\]
For all $j\in \m J$, define $\delta^{(0)}_j = \sigma_{\mx}\sqrt{\frac{2t}{n}}$ and 
\begin{align}
\label{eq:delta_k}
\delta^{(k)}_j = \frac{12}{5}\sigma_{0,j}\sqrt{\frac{2t}{n}} + 6^{-k}\l( \sigma_{\mx}\sqrt{\frac{2t}{n}} -  \frac{12}{5}\sigma_{0,j}\sqrt{\frac{2t}{n}} \r)
\end{align}
for $k\geq 1$.
Next, for each $j\in \m J$, we define 
\begin{align}
\label{eq:prelim}
T^{(0)}_{n,j} := T_n^{(0)} = \frac{1}{n\theta}\sum_{i=1}^n \psi\l( \theta Y_i \r),
\end{align}
(independent of $j$)\footnote{Particular choice of $T_n^{(0)}$ does not matter as long as $\| T_{n}^{(0)} - \mb EY\|$ is small with high probability.}, and 
\begin{align*}
T^{(k)}_{n,j} :=
T^{(k-1)}_{n,j} + \frac{1}{n\theta_j}\sum_{i=1}^n \psi\l( \theta_j \l( Y_i - T^{(k-1)}_{n,j} \r) \r)
\end{align*}
for $k\geq 1$. 
Finally, we apply Lepski's method to the collection of estimators $\l\{ T^{(k)}_{n,j}: j \in \m J \r\}$. 
To this end, define $\hat T_k := T^{(k)}_{n,j_k^\ast}$, where 
\[
j_k^\ast = \min\l\{ j\in \m J: \forall l\in \m J \text{ s.t. } l>j, \ \l\|  T^{(k)}_{n,l} - T^{(k)}_{n,j} \r\| \leq 2 \delta_l^{(k)}\r\}.
\]
Note that the estimator $\hat T_k$ is completely data-dependent. 
We are ready to state the main result of this section:

\begin{theorem}
\label{th:iterative}
Let
\[
\tau = 1.1K \sqrt{\frac{d^2+L t}{n}} + \sqrt{\frac{2t}{n}}\frac{1}{2\sigma_0}, 
\] 
where $K>0$ is an absolute constant, and assume that $\tau\leq 1/6$. 
Moreover, assume that 
\begin{align}
\label{eq:e_k}
\l(\frac{24}{5}\sigma_{0,\mx}\vee \sigma_{\mx}\r)\sqrt{\frac{2t}{n}}\leq 1.
\end{align}
Then for all $k\geq 0$ simultaneously, 
\[
\l\| \hat T_k - \mb EY \r\| \leq 3\l[(1-6^{-k})\frac{24}{5}\sigma_{0}\sqrt{\frac{2t}{n}} + 6^{-k}\sigma_{\mx}\sqrt{\frac{2t}{n}}\r]
\] 
with probability 
$\geq 1 - 8d\l( 1 + 2\log_2\l( \frac{12\sigma_{\mx}}{5\sigma_{0,\mn}}\r)\r)\log_2\l( \frac{2\sigma_{0,\mx}}{\sigma_{0,\mn}} \r) e^{-t}$.
\end{theorem}
\begin{proof}
See section 5 in appendix \ref{proof:iterative}.
\end{proof}
The next corollary easily follows from the preceding result.  
Let $\m A$ be the event of probability 
\[
\Pr(\m A)\geq 1 - 8d\l( 1 + 2\log_2\l( \frac{12\sigma_{\mx}}{5\sigma_{0,\mn}}\r)\r)\log_2\l( \frac{2\sigma_{0,\mx}}{\sigma_{0,\mn}} \r) e^{-t}
\] 
defined in Theorem \ref{th:iterative}. 
Since by the properties of the steepest descent scheme $T_{n,j}^{(k)}$ converges to the solution (denoted $\wh T_{\theta_j}^\ast$) of the problem \eqref{eq:matrix-catoni}, we can easily deduce the following inequality. 
\begin{corollary}
Let $\l\{ \wh T_{\theta_j}^\ast \r\}_{j\in \m J}$ satisfy the equations 
\[
\frac{1}{n\theta_j}\sum_{i=1}^n \psi\l(\theta_j (Y_i - \wh T_{\theta_j}^\ast)\r) = 0_{d\times d}, \quad j\in \m J.
\]
Then on event $\m A$, 
$\l\|  \wh T_{\theta_j}^\ast - \mb EY \r\| \leq \lim_{k\to \infty}\delta_j^{(k)}=\frac{12}{5}\sigma_{0,j}\sqrt{\frac{2t}{n}}$. 
\end{corollary}
One can further apply Lepski's method (see section \ref{section:variance}) to the collection $\l\{ \wh T_{\theta_j}^\ast \r\}_{j\in \m J}$ to obtain a completely data-dependent estimator 
$\wh T^\ast$ that satisfies 
\[
\l\| \wh T^\ast - \mb EY \r\| \leq \frac{72}{5}\sigma_{0}\sqrt{\frac{2t}{n}} 
\]
with high probability (in particular, on event $\m A$).

\section{Numerical simulation results}
\label{section:simulation}

Numerical simulation was performed for covariance estimation problem. 
Data was simulated as follows: let $U=\l(U^{(1)},\ldots,U^{(100)}\r)^T\in \mb R^{100}$ be a vector with i.i.d. coordinates such that 
$U^{(j)}\overset{\mathrm{d}}{=}\frac{1}{\sqrt{2 c(q)}}\l( \xi_{j,1}-\xi_{j,2} \r)$, where $\xi_{j,1}$ and $\xi_{j,2}, \ j=1,\ldots,100$, are independent random variables with probability density function 
\[
p_\xi(t;q)=\frac{q}{(1+t)^{1+q}}I\{t\geq 0\}
\] 
(which belongs to the Pareto family), $c(q)=\var(\xi)=\frac{q}{(q-1)^2 (q-2)}$ and $q=4.01$; in particular, $\var(U^{(j)})=1$. 
Finally, let $Z=\sqrt{\Sigma}U$, where $\Sigma$ is a diagonal matrix with $\Sigma_{11}=10, \ \Sigma_{22}=5, \ \Sigma_{33}=1$, 
and $\Sigma_{jj}=\frac{1}{97}, \ j\geq 4$. 
In particular, $\mb EZ=0$ and $\mb EZ Z^T=\Sigma$. 

The goal of numerical experiment was to evaluate the quality of estimation of the covariance matrix $\Sigma$ as well as its first eigenvector $e_1$ corresponding to $\lambda_1=10$. 
We tested two scenarios with sample sizes equal $n$ to $100$ and $1000$. 
In both cases, we generated $Z_1,\ldots,Z_n$, i.i.d. copies of $Z$, and centered the data via the spatial (or geometric) median defined as
\[
\widehat M_n=\argmin_{y\in \mb R^{100}}\sum_{j=1}^{100} \l\| y - Z_j\r\|_2.
\]
We compared two estimators, $\widehat S_n$ and $\widehat \Sigma_n$ constructed as follows: set $Z_j^0:=Z_j - \widehat M_n$ for brevity, and 
\[
\widehat S_n=\frac{1}{n}\sum_{j=1}^n Z_j^0 {Z_j^0}^T,
\]
which is the analogue of sample covariance with ``robust centering''. 

Next, $\widehat \Sigma_n$ was constructed using a version of Lepski's method described in section \ref{section:variance}. 
We provide details for completeness: set 
\[
\sigma_{\mx}:=2\sqrt{\l\| \frac{1}{n}\sum_{j=1}^n \|Z_j^0\|_2^2 Z_j^0 {Z_j^0}^T \r\|}, \ 
\sigma_{\mn}=\frac{\sigma_{\mx}}{100},
\] 
\[
\m J=\l\{ j\in \mb Z: \  \sigma_{\mn} < 1.3^j  \leq \sigma_{\mx} \r\},
\] 
and let $\psi(\cdot)$ be the function defined in (\ref{eq:psi2}). 
Let $t=\log 10$, and for $j\in \m J$, set $\theta_j=\sqrt{\frac{2t}{n}}\frac{1}{1.3^j}$ and 
$
\hat\Sigma_{n,j} = \frac{1}{n\theta_j}\sum_{i=1}^n \psi\l(\theta_j Z_i^0 {Z_i^0}^T\r).
$
Finally, define 
\[
j_\ast:=\min\l\{j\in \m J: \ \forall k>j, \ \| \hat\Sigma_{n,k} - \hat\Sigma_{n,j} \|\leq 1.3^{k}\sqrt{\frac{t}{n}}  \r\}
\]
(note that we modified some constants compared to the ``theoretical'' version), and finally set 
$\widehat \Sigma_n:=\hat\Sigma_{n,j_\ast}$. 

Quality of covariance estimation was evaluated via comparing $\frac{\|\widehat S_n - \Sigma\|}{\|\Sigma\|}$ with 
$\frac{\|\widehat \Sigma_n - \Sigma\|}{\| \Sigma \|}$ over 500 runs of simulations. 
We also compared errors of estimation of projectors onto the first principal component, 
\[
\l\| u_1(\widehat S_n)u_1(\widehat S_n)^T - u_1(\Sigma)u_1(\Sigma)^T \r\| 
\text{  and  } \l\| u_1(\widehat \Sigma_n)u_1(\widehat \Sigma_n)^T - u_1(\Sigma)u_1(\Sigma)^T \r\|,
\]
where $u_1(\cdot)$ denotes the eigenvector corresponding to the largest eigenvalue of a matrix. 
Histograms illustrating performance of both estimators are presented in figures \ref{cov100-100} and \ref{pca100-100} (for the sample size $n=100$), and in figures \ref{cov100-1000} and \ref{pca100-1000} (for the sample size $n=1000$). 
It is clear from the graphs that in all scenarios, $\widehat \Sigma_n$ performs significantly better than $\widehat S_n$. 

\begin{figure}[t]
  \centering
  \subfloat[Covariance matrix estimation error]{
    \includegraphics[width=0.55\textwidth]{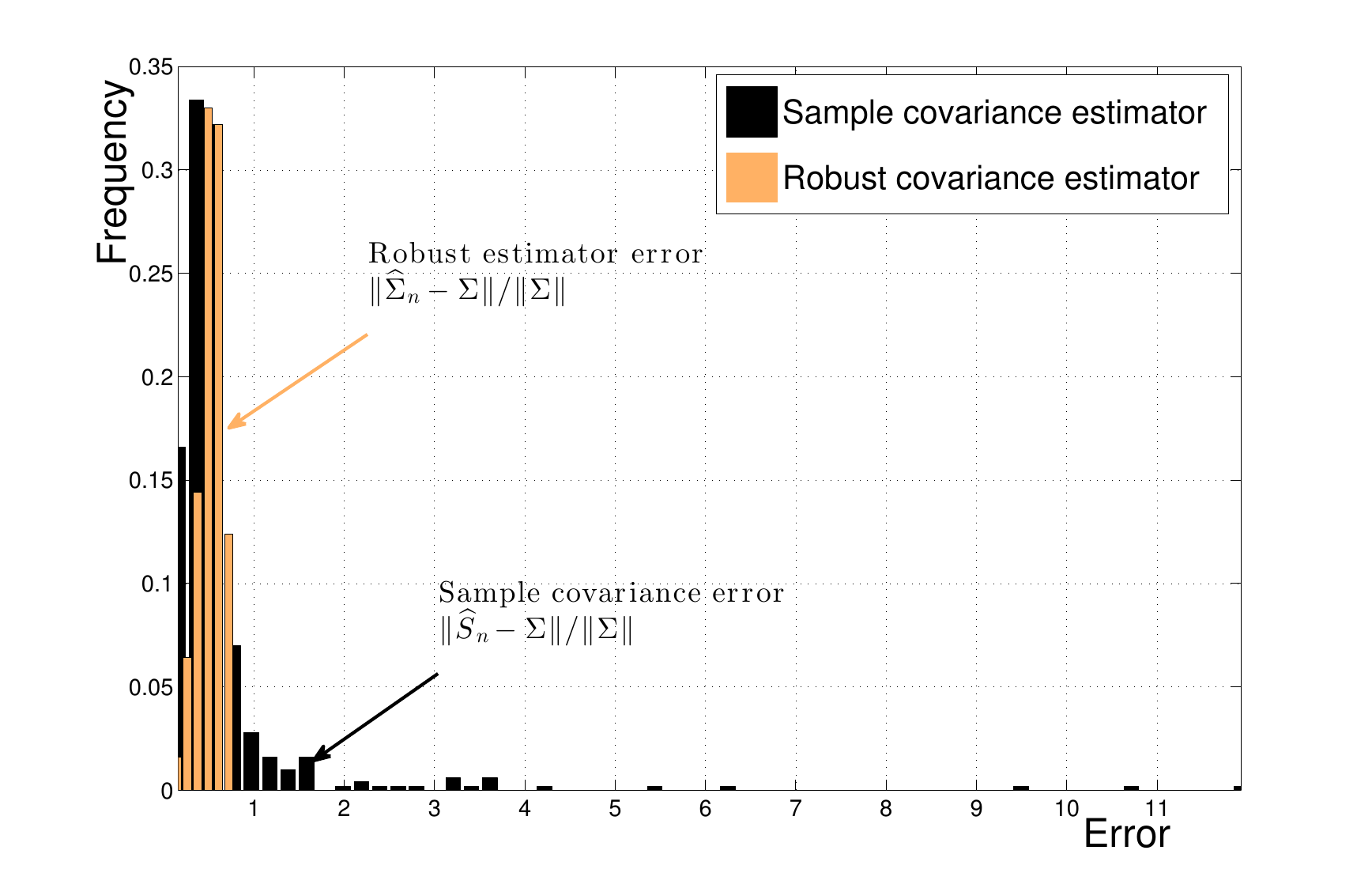}
    \label{cov100-100}}
  \subfloat[First principal component estimation error]{
    \includegraphics[width=0.55\textwidth]{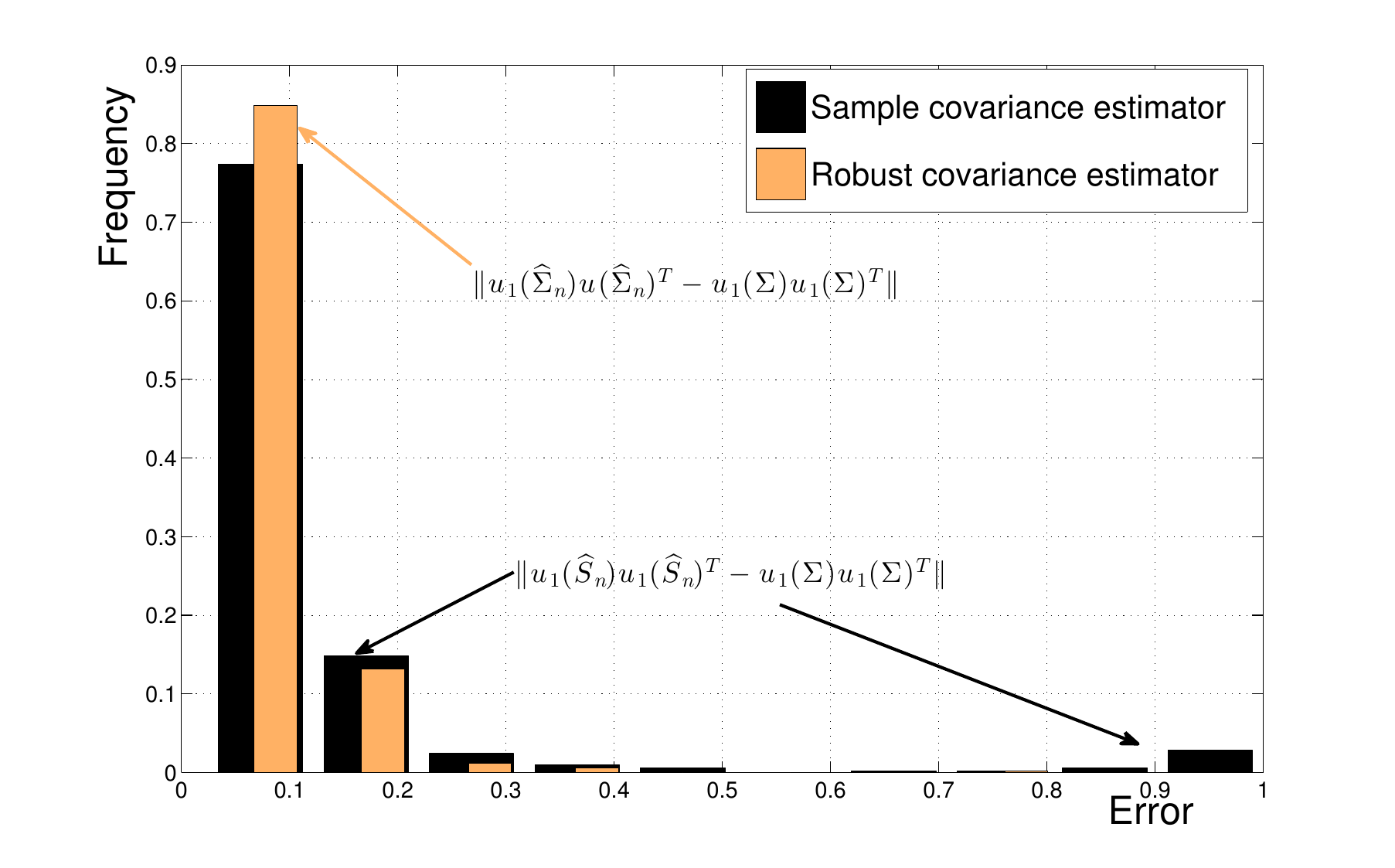}
    \label{pca100-100}}
  \caption{Sample size $n = 100$, dimension $d = 100$.}
\end{figure}

\begin{figure}[t]
  \centering
  \subfloat[Covariance matrix estimation error]{
    \includegraphics[width=0.55\textwidth]{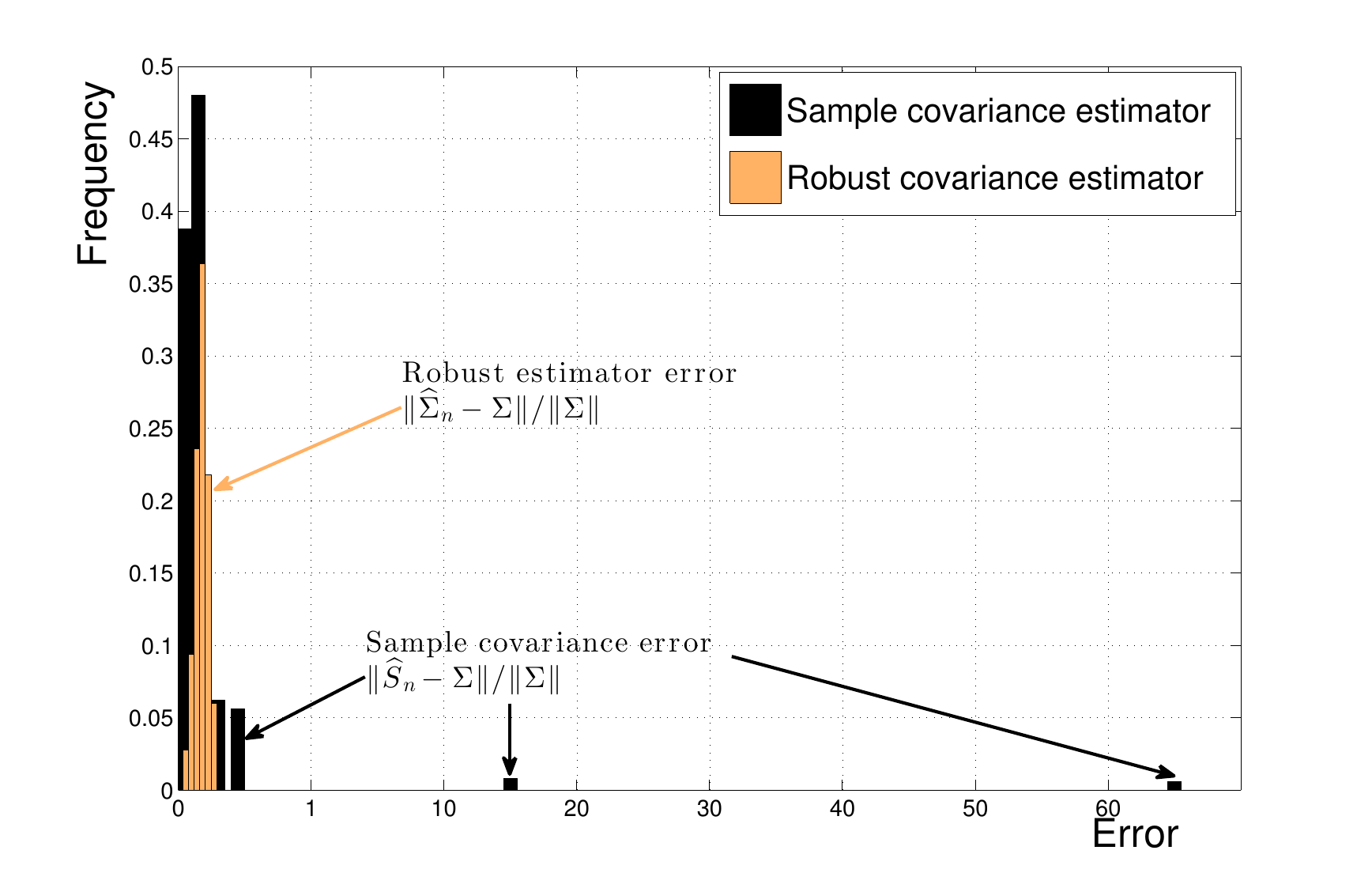}
    \label{cov100-1000}}
  \subfloat[First principal component estimation error]{
    \includegraphics[width=0.55\textwidth]{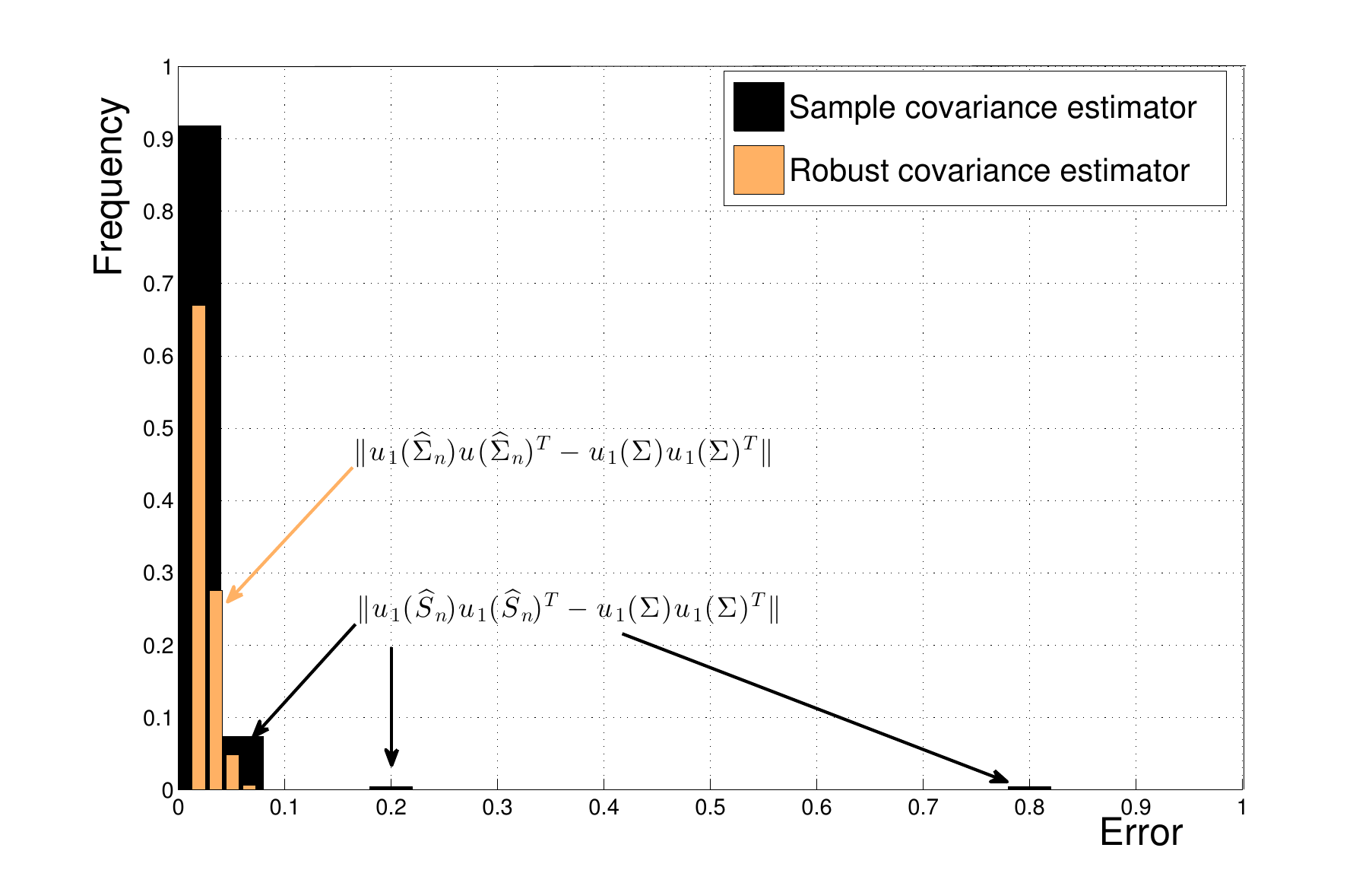}
    \label{pca100-1000}}
  \caption{Sample size $n = 1000$, dimension $d = 100$.}
\end{figure}

\section*{Acknowledgements}
I want to thank L. Goldstein, A. Juditsky, A. Nemirovski, as well as the anonymous Referees and the Associate Editor for their insightful suggestions that helped to improve the quality of presentation.

\bibliographystyle{imsart-nameyear}
\bibliography{bibliography}

\begin{thebibliography}{49}

\bibitem[\protect\citeauthoryear{Ahlswede and
  Winter}{2002}]{Ahlswede2002Strong-converse00}
\begin{barticle}[author]
\bauthor{\bsnm{Ahlswede},~\bfnm{R.}\binits{R.}} \AND
  \bauthor{\bsnm{Winter},~\bfnm{A.}\binits{A.}}
(\byear{2002}).
\btitle{Strong converse for identification via quantum channels}.
\bjournal{IEEE Trans. Inform. Theory}
\bvolume{48}
\bpages{569--579}.
\bdoi{10.1109/18.985947}
\end{barticle}
\endbibitem

\bibitem[\protect\citeauthoryear{Aleksandrov and
  Peller}{2016}]{aleksandrov2016operator}
\begin{barticle}[author]
\bauthor{\bsnm{Aleksandrov},~\bfnm{Alexei~Borisovich}\binits{A.~B.}} \AND
  \bauthor{\bsnm{Peller},~\bfnm{Vladimir~Vsevolodovich}\binits{V.~V.}}
(\byear{2016}).
\btitle{Operator Lipschitz functions}.
\bjournal{Russian Mathematical Surveys}
\bvolume{71}
\bpages{605}.
\end{barticle}
\endbibitem

\bibitem[\protect\citeauthoryear{Alon, Matias and
  Szegedy}{1996}]{alon1996space}
\begin{binproceedings}[author]
\bauthor{\bsnm{Alon},~\bfnm{N.}\binits{N.}},
  \bauthor{\bsnm{Matias},~\bfnm{Y.}\binits{Y.}} \AND
  \bauthor{\bsnm{Szegedy},~\bfnm{M.}\binits{M.}}
(\byear{1996}).
\btitle{The space complexity of approximating the frequency moments}.
In \bbooktitle{Proceedings of the twenty-eighth annual ACM symposium on Theory
  of computing}
\bpages{20--29}.
\bpublisher{ACM}.
\end{binproceedings}
\endbibitem

\bibitem[\protect\citeauthoryear{Bhatia}{1997}]{bhatia1997matrix}
\begin{barticle}[author]
\bauthor{\bsnm{Bhatia},~\bfnm{R.}\binits{R.}}
(\byear{1997}).
\btitle{Matrix analysis}.
\bjournal{Springer}.
\end{barticle}
\endbibitem

\bibitem[\protect\citeauthoryear{Boucheron, Lugosi and
  Massart}{2013}]{boucheron2013concentration}
\begin{bbook}[author]
\bauthor{\bsnm{Boucheron},~\bfnm{St{\'e}phane}\binits{S.}},
  \bauthor{\bsnm{Lugosi},~\bfnm{G{\'a}bor}\binits{G.}} \AND
  \bauthor{\bsnm{Massart},~\bfnm{Pascal}\binits{P.}}
(\byear{2013}).
\btitle{Concentration inequalities: A nonasymptotic theory of independence}.
\bpublisher{Oxford university press}.
\end{bbook}
\endbibitem

\bibitem[\protect\citeauthoryear{Brownlees, Joly and
  Lugosi}{2015}]{brownlees2015empirical}
\begin{barticle}[author]
\bauthor{\bsnm{Brownlees},~\bfnm{C.}\binits{C.}},
  \bauthor{\bsnm{Joly},~\bfnm{E.}\binits{E.}} \AND
  \bauthor{\bsnm{Lugosi},~\bfnm{G.}\binits{G.}}
(\byear{2015}).
\btitle{Empirical risk minimization for heavy-tailed losses}.
\bjournal{The Annals of Statistics}
\bvolume{43}
\bpages{2507--2536}.
\end{barticle}
\endbibitem

\bibitem[\protect\citeauthoryear{Butler, Davies and
  Jhun}{1993}]{butler1993asymptotics}
\begin{barticle}[author]
\bauthor{\bsnm{Butler},~\bfnm{R.~W.}\binits{R.~W.}},
  \bauthor{\bsnm{Davies},~\bfnm{P.~L.}\binits{P.~L.}} \AND
  \bauthor{\bsnm{Jhun},~\bfnm{M.}\binits{M.}}
(\byear{1993}).
\btitle{Asymptotics for the minimum covariance determinant estimator}.
\bjournal{The Annals of Statistics}
\bpages{1385--1400}.
\end{barticle}
\endbibitem

\bibitem[\protect\citeauthoryear{Cai, Ren and Zhou}{2016}]{cai2016}
\begin{barticle}[author]
\bauthor{\bsnm{Cai},~\bfnm{T.~T.}\binits{T.~T.}},
  \bauthor{\bsnm{Ren},~\bfnm{Z.}\binits{Z.}} \AND
  \bauthor{\bsnm{Zhou},~\bfnm{H.~H.}\binits{H.~H.}}
(\byear{2016}).
\btitle{Estimating structured high-dimensional covariance and precision
  matrices: optimal rates and adaptive estimation}.
\bjournal{Electron. J. Statist.}
\bvolume{10}
\bpages{1--59}.
\bdoi{10.1214/15-EJS1081}
\end{barticle}
\endbibitem

\bibitem[\protect\citeauthoryear{Cai, Zhang and Zhou}{2010}]{cai2010optimal}
\begin{barticle}[author]
\bauthor{\bsnm{Cai},~\bfnm{T.~T.}\binits{T.~T.}},
  \bauthor{\bsnm{Zhang},~\bfnm{C.~H.}\binits{C.~H.}} \AND
  \bauthor{\bsnm{Zhou},~\bfnm{H.~H.}\binits{H.~H.}}
(\byear{2010}).
\btitle{Optimal rates of convergence for covariance matrix estimation}.
\bjournal{The Annals of Statistics}
\bvolume{38}
\bpages{2118--2144}.
\end{barticle}
\endbibitem

\bibitem[\protect\citeauthoryear{Cand{\`e}s et~al.}{2011}]{candes2011robust}
\begin{barticle}[author]
\bauthor{\bsnm{Cand{\`e}s},~\bfnm{E.~J.}\binits{E.~J.}},
  \bauthor{\bsnm{Li},~\bfnm{X.}\binits{X.}},
  \bauthor{\bsnm{Ma},~\bfnm{Y.}\binits{Y.}} \AND
  \bauthor{\bsnm{Wright},~\bfnm{J.}\binits{J.}}
(\byear{2011}).
\btitle{Robust principal component analysis?}
\bjournal{Journal of the ACM (JACM)}
\bvolume{58}
\bpages{11}.
\end{barticle}
\endbibitem

\bibitem[\protect\citeauthoryear{Carlen}{2010}]{carlen2010trace}
\begin{bunpublished}[author]
\bauthor{\bsnm{Carlen},~\bfnm{E.}\binits{E.}}
(\byear{2010}).
\btitle{Trace inequalities and quantum entropy: an introductory course}.
\bnote{{A}vailable at
  \url{http://www.mathphys.org/AZschool/material/AZ09-carlen.pdf}}.
\end{bunpublished}
\endbibitem

\bibitem[\protect\citeauthoryear{Catoni}{2012}]{catoni2012challenging}
\begin{binproceedings}[author]
\bauthor{\bsnm{Catoni},~\bfnm{O.}\binits{O.}}
(\byear{2012}).
\btitle{Challenging the empirical mean and empirical variance: a deviation
  study}.
In \bbooktitle{Annales de l'Institut Henri Poincar{\'e}, Probabilit{\'e}s et
  Statistiques}
\bvolume{48}
\bpages{1148--1185}.
\end{binproceedings}
\endbibitem

\bibitem[\protect\citeauthoryear{Catoni}{2016}]{catoni2016pac}
\begin{barticle}[author]
\bauthor{\bsnm{Catoni},~\bfnm{O.}\binits{O.}}
(\byear{2016}).
\btitle{{PAC}-{B}ayesian bounds for the {G}ram matrix and least squares
  regression with a random design}.
\bjournal{arXiv preprint arXiv:1603.05229}.
\end{barticle}
\endbibitem

\bibitem[\protect\citeauthoryear{Davies}{1992}]{davies1992asymptotics}
\begin{barticle}[author]
\bauthor{\bsnm{Davies},~\bfnm{L.}\binits{L.}}
(\byear{1992}).
\btitle{The asymptotics of {Rousseeuw's} minimum volume ellipsoid estimator}.
\bjournal{The Annals of Statistics}
\bpages{1828--1843}.
\end{barticle}
\endbibitem

\bibitem[\protect\citeauthoryear{Devroye et~al.}{2015}]{devroye2015sub}
\begin{barticle}[author]
\bauthor{\bsnm{Devroye},~\bfnm{L.}\binits{L.}},
  \bauthor{\bsnm{Lerasle},~\bfnm{M.}\binits{M.}},
  \bauthor{\bsnm{Lugosi},~\bfnm{G.}\binits{G.}} \AND
  \bauthor{\bsnm{Oliveira},~\bfnm{R.~I.}\binits{R.~I.}}
(\byear{2015}).
\btitle{Sub-{G}aussian mean estimators}.
\bjournal{arXiv preprint arXiv:1509.05845}.
\end{barticle}
\endbibitem

\bibitem[\protect\citeauthoryear{Dirksen}{2013}]{dirksen2013tail}
\begin{barticle}[author]
\bauthor{\bsnm{Dirksen},~\bfnm{S.}\binits{S.}}
(\byear{2013}).
\btitle{Tail bounds via generic chaining}.
\bjournal{arXiv preprint arXiv:1309.3522}.
\end{barticle}
\endbibitem

\bibitem[\protect\citeauthoryear{Fan, Wang and
  Zhong}{2016}]{fan2016eigenvector}
\begin{barticle}[author]
\bauthor{\bsnm{Fan},~\bfnm{Jianqing}\binits{J.}},
  \bauthor{\bsnm{Wang},~\bfnm{Weichen}\binits{W.}} \AND
  \bauthor{\bsnm{Zhong},~\bfnm{Yiqiao}\binits{Y.}}
(\byear{2016}).
\btitle{An $\ell_\infty$ Eigenvector Perturbation Bound and Its Application to
  Robust Covariance Estimation}.
\bjournal{arXiv preprint arXiv:1603.03516}.
\end{barticle}
\endbibitem

\bibitem[\protect\citeauthoryear{Fan, Wang and Zhu}{2016}]{fan2016robust}
\begin{barticle}[author]
\bauthor{\bsnm{Fan},~\bfnm{J.}\binits{J.}},
  \bauthor{\bsnm{Wang},~\bfnm{W.}\binits{W.}} \AND
  \bauthor{\bsnm{Zhu},~\bfnm{Z.}\binits{Z.}}
(\byear{2016}).
\btitle{Robust Low-Rank Matrix Recovery}.
\bjournal{arXiv preprint arXiv:1603.08315}.
\end{barticle}
\endbibitem

\bibitem[\protect\citeauthoryear{Giulini}{2015}]{giulini2015pac}
\begin{barticle}[author]
\bauthor{\bsnm{Giulini},~\bfnm{I.}\binits{I.}}
(\byear{2015}).
\btitle{{PAC-Bayesian} bounds for {Principal Component Analysis} in {Hilbert}
  spaces}.
\bjournal{arXiv preprint arXiv:1511.06263}.
\end{barticle}
\endbibitem

\bibitem[\protect\citeauthoryear{Hsu and Sabato}{2016}]{hsu2013loss}
\begin{barticle}[author]
\bauthor{\bsnm{Hsu},~\bfnm{D.}\binits{D.}} \AND
  \bauthor{\bsnm{Sabato},~\bfnm{S.}\binits{S.}}
(\byear{2016}).
\btitle{Loss Minimization and Parameter Estimation with Heavy Tails}.
\bjournal{Journal of Machine Learning Research}
\bvolume{17}
\bpages{1-40}.
\end{barticle}
\endbibitem

\bibitem[\protect\citeauthoryear{Huber}{1964}]{huber1964robust}
\begin{barticle}[author]
\bauthor{\bsnm{Huber},~\bfnm{P.~J.}\binits{P.~J.}}
(\byear{1964}).
\btitle{Robust estimation of a location parameter}.
\bjournal{The Annals of Mathematical Statistics}
\bvolume{35}
\bpages{73--101}.
\end{barticle}
\endbibitem

\bibitem[\protect\citeauthoryear{Huber and
  Ronchetti}{2009}]{Huber2009Robust-statisti00}
\begin{bbook}[author]
\bauthor{\bsnm{Huber},~\bfnm{P.~J.}\binits{P.~J.}} \AND
  \bauthor{\bsnm{Ronchetti},~\bfnm{E.~M.}\binits{E.~M.}}
(\byear{2009}).
\btitle{Robust statistics},
\bedition{second} ed.
\bseries{Wiley Series in Probability and Statistics}.
\bpublisher{John Wiley \& Sons Inc.}, \baddress{Hoboken, NJ}.
\bdoi{10.1002/9780470434697}
\end{bbook}
\endbibitem

\bibitem[\protect\citeauthoryear{Hubert, Rousseeuw and
  Van~Aelst}{2008}]{hubert2008high}
\begin{barticle}[author]
\bauthor{\bsnm{Hubert},~\bfnm{M.}\binits{M.}},
  \bauthor{\bsnm{Rousseeuw},~\bfnm{P.~J.}\binits{P.~J.}} \AND
  \bauthor{\bsnm{Van~Aelst},~\bfnm{S.}\binits{S.}}
(\byear{2008}).
\btitle{High-breakdown robust multivariate methods}.
\bjournal{Statistical Science}
\bpages{92--119}.
\end{barticle}
\endbibitem

\bibitem[\protect\citeauthoryear{Jerrum, Valiant and
  Vazirani}{1986}]{jerrum1986random}
\begin{barticle}[author]
\bauthor{\bsnm{Jerrum},~\bfnm{Mark~R}\binits{M.~R.}},
  \bauthor{\bsnm{Valiant},~\bfnm{Leslie~G}\binits{L.~G.}} \AND
  \bauthor{\bsnm{Vazirani},~\bfnm{Vijay~V}\binits{V.~V.}}
(\byear{1986}).
\btitle{Random generation of combinatorial structures from a uniform
  distribution}.
\bjournal{Theoretical Computer Science}
\bvolume{43}
\bpages{169--188}.
\end{barticle}
\endbibitem

\bibitem[\protect\citeauthoryear{Joly et~al.}{2017}]{joly2017estimation}
\begin{barticle}[author]
\bauthor{\bsnm{Joly},~\bfnm{Emilien}\binits{E.}},
  \bauthor{\bsnm{Lugosi},~\bfnm{G{\'a}bor}\binits{G.}},
  \bauthor{\bsnm{Oliveira},~\bfnm{Roberto~Imbuzeiro}\binits{R.~I.}}
  \betal{et~al.}
(\byear{2017}).
\btitle{On the estimation of the mean of a random vector}.
\bjournal{Electronic Journal of Statistics}
\bvolume{11}
\bpages{440--451}.
\end{barticle}
\endbibitem

\bibitem[\protect\citeauthoryear{Klopp, Lounici and
  Tsybakov}{2014}]{klopp2014robust}
\begin{barticle}[author]
\bauthor{\bsnm{Klopp},~\bfnm{O.}\binits{O.}},
  \bauthor{\bsnm{Lounici},~\bfnm{K.}\binits{K.}} \AND
  \bauthor{\bsnm{Tsybakov},~\bfnm{A.~B.}\binits{A.~B.}}
(\byear{2014}).
\btitle{Robust matrix completion}.
\bjournal{arXiv preprint arXiv:1412.8132}.
\end{barticle}
\endbibitem

\bibitem[\protect\citeauthoryear{Koltchinskii, Lounici and
  Tsybakov}{2011}]{koltchinskii2011nuclear}
\begin{barticle}[author]
\bauthor{\bsnm{Koltchinskii},~\bfnm{V.}\binits{V.}},
  \bauthor{\bsnm{Lounici},~\bfnm{K.}\binits{K.}} \AND
  \bauthor{\bsnm{Tsybakov},~\bfnm{A.~B.}\binits{A.~B.}}
(\byear{2011}).
\btitle{Nuclear-norm penalization and optimal rates for noisy low-rank matrix
  completion}.
\bjournal{The Annals of Statistics}
\bvolume{39}
\bpages{2302--2329}.
\end{barticle}
\endbibitem

\bibitem[\protect\citeauthoryear{Koltchinskii and
  Lounici}{2016}]{koltchinskii2016new}
\begin{barticle}[author]
\bauthor{\bsnm{Koltchinskii},~\bfnm{V.}\binits{V.}} \AND
  \bauthor{\bsnm{Lounici},~\bfnm{K.}\binits{K.}}
(\byear{2016}).
\btitle{New asymptotic results in principal component analysis}.
\bjournal{arXiv preprint arXiv:1601.01457}.
\end{barticle}
\endbibitem

\bibitem[\protect\citeauthoryear{Koltchinskii and
  Lounici}{2017}]{koltchinskii2017concentration}
\begin{barticle}[author]
\bauthor{\bsnm{Koltchinskii},~\bfnm{Vladimir}\binits{V.}} \AND
  \bauthor{\bsnm{Lounici},~\bfnm{Karim}\binits{K.}}
(\byear{2017}).
\btitle{Concentration inequalities and moment bounds for sample covariance
  operators}.
\bjournal{Bernoulli}
\bvolume{23}
\bpages{110--133}.
\end{barticle}
\endbibitem

\bibitem[\protect\citeauthoryear{Lam and Fan}{2009}]{lam2009sparsistency}
\begin{barticle}[author]
\bauthor{\bsnm{Lam},~\bfnm{Clifford}\binits{C.}} \AND
  \bauthor{\bsnm{Fan},~\bfnm{Jianqing}\binits{J.}}
(\byear{2009}).
\btitle{Sparsistency and rates of convergence in large covariance matrix
  estimation}.
\bjournal{Annals of statistics}
\bvolume{37}
\bpages{4254}.
\end{barticle}
\endbibitem

\bibitem[\protect\citeauthoryear{Lepski}{1992}]{lepskii1992asymptotically}
\begin{barticle}[author]
\bauthor{\bsnm{Lepski},~\bfnm{O.}\binits{O.}}
(\byear{1992}).
\btitle{Asymptotically minimax adaptive estimation. {I}: {Upper bounds.
  Optimally adaptive estimates}}.
\bjournal{Theory of Probability \& Its Applications}
\bvolume{36}
\bpages{682--697}.
\end{barticle}
\endbibitem

\bibitem[\protect\citeauthoryear{Lerasle and
  Oliveira}{2011}]{lerasle2011robust}
\begin{barticle}[author]
\bauthor{\bsnm{Lerasle},~\bfnm{Matthieu}\binits{M.}} \AND
  \bauthor{\bsnm{Oliveira},~\bfnm{Roberto~I}\binits{R.~I.}}
(\byear{2011}).
\btitle{Robust empirical mean estimators}.
\bjournal{arXiv preprint arXiv:1112.3914}.
\end{barticle}
\endbibitem

\bibitem[\protect\citeauthoryear{Lieb}{1973}]{lieb1}
\begin{barticle}[author]
\bauthor{\bsnm{Lieb},~\bfnm{E.~H.}\binits{E.~H.}}
(\byear{1973}).
\btitle{Convex trace functions and the {W}igner-{Y}anase-{D}yson conjecture}.
\bjournal{Advances in Math.}
\bvolume{11}
\bpages{267--288}.
\end{barticle}
\endbibitem

\bibitem[\protect\citeauthoryear{Lounici}{2014}]{lounici2014high}
\begin{barticle}[author]
\bauthor{\bsnm{Lounici},~\bfnm{K.}\binits{K.}}
(\byear{2014}).
\btitle{High-dimensional covariance matrix estimation with missing
  observations}.
\bjournal{Bernoulli}
\bvolume{20}
\bpages{1029--1058}.
\end{barticle}
\endbibitem

\bibitem[\protect\citeauthoryear{Lugosi and Mendelson}{2017}]{lugosi2017sub}
\begin{barticle}[author]
\bauthor{\bsnm{Lugosi},~\bfnm{G\'{a}bor}\binits{G.}} \AND
  \bauthor{\bsnm{Mendelson},~\bfnm{Shahar}\binits{S.}}
(\byear{2017}).
\btitle{{Sub-Gaussian} estimators of the mean of a random vector}.
\bjournal{arXiv preprint arXiv:1702.00482}.
\end{barticle}
\endbibitem

\bibitem[\protect\citeauthoryear{Maronna}{1976}]{maronna1976}
\begin{barticle}[author]
\bauthor{\bsnm{Maronna},~\bfnm{R.~A.}\binits{R.~A.}}
(\byear{1976}).
\btitle{Robust {M}-Estimators of Multivariate Location and Scatter}.
\bjournal{Ann. Statist.}
\bvolume{4}
\bpages{51--67}.
\end{barticle}
\endbibitem

\bibitem[\protect\citeauthoryear{Minsker}{2015}]{minsker2015geometric}
\begin{barticle}[author]
\bauthor{\bsnm{Minsker},~\bfnm{S.}\binits{S.}}
(\byear{2015}).
\btitle{Geometric median and robust estimation in {B}anach spaces}.
\bjournal{Bernoulli}
\bvolume{21}
\bpages{2308--2335}.
\end{barticle}
\endbibitem

\bibitem[\protect\citeauthoryear{Minsker and Wei}{2017}]{S.-Minsker:2017aa}
\begin{barticle}[author]
\bauthor{\bsnm{Minsker},~\bfnm{S.}\binits{S.}} \AND
  \bauthor{\bsnm{Wei},~\bfnm{X.}\binits{X.}}
(\byear{2017}).
\btitle{Estimation of the covariance structure of heavy-tailed distributions}.
\bjournal{arXiv preprint arXiv:1708.00502}.
\end{barticle}
\endbibitem

\bibitem[\protect\citeauthoryear{Nemirovski and
  Yudin}{1983}]{Nemirovski1983Problem-complex00}
\begin{bbook}[author]
\bauthor{\bsnm{Nemirovski},~\bfnm{A.}\binits{A.}} \AND
  \bauthor{\bsnm{Yudin},~\bfnm{D.}\binits{D.}}
(\byear{1983}).
\btitle{Problem complexity and method efficiency in optimization}.
\bpublisher{John Wiley \& Sons Inc.}
\end{bbook}
\endbibitem

\bibitem[\protect\citeauthoryear{Oliveira}{2009}]{oliveira2009concentration}
\begin{barticle}[author]
\bauthor{\bsnm{Oliveira},~\bfnm{R.~I.}\binits{R.~I.}}
(\byear{2009}).
\btitle{Concentration of the adjacency matrix and of the {L}aplacian in random
  graphs with independent edges}.
\bjournal{arXiv preprint arXiv:0911.0600}.
\end{barticle}
\endbibitem

\bibitem[\protect\citeauthoryear{Srivastava and
  Vershynin}{2013}]{srivastava2013covariance}
\begin{barticle}[author]
\bauthor{\bsnm{Srivastava},~\bfnm{N.}\binits{N.}} \AND
  \bauthor{\bsnm{Vershynin},~\bfnm{R.}\binits{R.}}
(\byear{2013}).
\btitle{Covariance estimation for distributions with $2+\varepsilon$ moments}.
\bjournal{The Annals of Probability}
\bvolume{41}
\bpages{3081--3111}.
\end{barticle}
\endbibitem

\bibitem[\protect\citeauthoryear{Talagrand}{2014}]{talagrand2014upper}
\begin{bbook}[author]
\bauthor{\bsnm{Talagrand},~\bfnm{M.}\binits{M.}}
(\byear{2014}).
\btitle{Upper and lower bounds for stochastic processes: modern methods and
  classical problems}
\bvolume{60}.
\bpublisher{Springer Science \& Business Media}.
\end{bbook}
\endbibitem

\bibitem[\protect\citeauthoryear{Tropp}{2012a}]{tropp1}
\begin{barticle}[author]
\bauthor{\bsnm{Tropp},~\bfnm{J.~A.}\binits{J.~A.}}
(\byear{2012}a).
\btitle{User-friendly tail bounds for sums of random matrices}.
\bjournal{Found. Comput. Math.}
\bvolume{12}
\bpages{389--434}.
\bdoi{10.1007/s10208-011-9099-z}
\end{barticle}
\endbibitem

\bibitem[\protect\citeauthoryear{Tropp}{2012b}]{tropp2012user}
\begin{barticle}[author]
\bauthor{\bsnm{Tropp},~\bfnm{J.~A.}\binits{J.~A.}}
(\byear{2012}b).
\btitle{User-friendly tail bounds for sums of random matrices}.
\bjournal{Foundations of computational mathematics}
\bvolume{12}
\bpages{389--434}.
\end{barticle}
\endbibitem

\bibitem[\protect\citeauthoryear{Tropp}{2015}]{tropp2015introduction}
\begin{barticle}[author]
\bauthor{\bsnm{Tropp},~\bfnm{J.~A.}\binits{J.~A.}}
(\byear{2015}).
\btitle{An introduction to matrix concentration inequalities}.
\bjournal{arXiv preprint arXiv:1501.01571}.
\end{barticle}
\endbibitem

\bibitem[\protect\citeauthoryear{Tyler}{1987}]{tyler1987distribution}
\begin{barticle}[author]
\bauthor{\bsnm{Tyler},~\bfnm{D.~E.}\binits{D.~E.}}
(\byear{1987}).
\btitle{A distribution-free {M}-estimator of multivariate scatter}.
\bjournal{The Annals of Statistics}
\bpages{234--251}.
\end{barticle}
\endbibitem

\bibitem[\protect\citeauthoryear{Vershynin}{2007}]{VershyninNon-Asymptotic-00}
\begin{bunpublished}[author]
\bauthor{\bsnm{Vershynin},~\bfnm{R.}\binits{R.}}
(\byear{2007}).
\btitle{Non-Asymptotic Theory of Random Matrices: lecture notes.}
\bnote{Available at
  \url{http://www-personal.umich.edu/~romanv/teaching/2006-07/280/lec6.pdf}}.
\end{bunpublished}
\endbibitem

\bibitem[\protect\citeauthoryear{Vershynin}{2010}]{vershynin2010introduction}
\begin{barticle}[author]
\bauthor{\bsnm{Vershynin},~\bfnm{R.}\binits{R.}}
(\byear{2010}).
\btitle{Introduction to the non-asymptotic analysis of random matrices}.
\bjournal{arXiv preprint arXiv:1011.3027}.
\end{barticle}
\endbibitem

\bibitem[\protect\citeauthoryear{Zhang, Cheng and
  Singer}{2016}]{zhang2016marvcenko}
\begin{barticle}[author]
\bauthor{\bsnm{Zhang},~\bfnm{T.}\binits{T.}},
  \bauthor{\bsnm{Cheng},~\bfnm{X.}\binits{X.}} \AND
  \bauthor{\bsnm{Singer},~\bfnm{A.}\binits{A.}}
(\byear{2016}).
\btitle{{Mar{\v{c}}enko-Pastur} law for {Tylers M-estimator}}.
\bjournal{Journal of Multivariate Analysis}
\bvolume{149}
\bpages{114--123}.
\end{barticle}
\endbibitem

\end{thebibliography}

\appendix

\section{Supplementary technical results}
\label{section:technical}

\begin{lemma}
\label{lemma:gradient}
Let $F:\mb R\mapsto \mb R$ be a continuously differentiable function, and $S\in \mb C^{d\times d}$ be a self-adjoint matrix. Then the gradient of $G(S):=\tr F(S)$ is 
\[
\nabla G(S) = F'(S),
\]  
where $F'$ is the derivative of $F$ and $F'(S): \mb C^{d\times d}\mapsto \mb C^{d\times d}$ is the matrix function in the sense of definition \ref{matrix-function}. 
\end{lemma}
\begin{proof}
We will first check the claim assuming that $F$ is a polynomial of the form $F(x)=x^k, \ k\in \mb N$. 
Let $H=H^\ast$ be a self-adjoint operator, and consider the directional derivative $d G(S;H)$ of $G$ in direction $H$:
\begin{align*}
dG(S;H)&=\lim_{t\to 0}\frac{1}{t} \tr \l(  (S+tH)^k - S^k\r)= 
\sum_{j=1}^{k} \tr \l( S^{j -1} H S^{k - j} \r) \\ 
&
=  \tr \l(k S^{k-1} H \r) = \dotp{F'(S)}{H}, 
\end{align*}
hence the claim holds for monomials. 
By linearity, it also holds for arbitrary polynomials. 
It remains to extend the claim to arbitrary continuously differentiable function via a standard approximation argument (for instance, see \citep[][chapter 5, section 3]{bhatia1997matrix}).

\end{proof}

\begin{lemma}
\label{lemma:psi_alpha}
Let $1<\alpha\leq 2$ and $c_\alpha=\frac{\alpha-1}{\alpha}\vee \sqrt{\frac{2-\alpha}{\alpha}}$. 
Then $1+y+c_\alpha|y|^\alpha>0$ and
\[
-\log(1+y+c_\alpha |y|^\alpha)\leq \log(1-y +c_\alpha |y|^{\alpha})  \quad \text{for all $y\in \mb R$.}
\]
\end{lemma}
\begin{proof}
To check the first claim, it is enough to note that $f(y)=1+y+c_\alpha|y|^\alpha$ is convex and its minimum is attained for 
$y_m=-\l(\frac{1}{\alpha c_\alpha} \r)^{1/(\alpha-1)}$. 
It is easy to check that $f(y_m)=1-y_m+\frac{y_m}{\alpha}$, which implies that $f(y_m)>0\iff c_\alpha>\frac{\alpha-1}{\alpha^2}$ which always holds since $c_\alpha\geq \frac{\alpha-1}{\alpha}$ and $\alpha>1$.

For the second part, it is enough to show that $(1+c_\alpha |y|^{\alpha} + y)(1+ c_\alpha |y|^{\alpha} - y)\geq 1$ for all $y\in \mb R$, which is equivalent to claiming that $c_\alpha^2 y^{2\alpha}+2c_\alpha y^\alpha\geq y^2, \ y\geq 0$.  
Note that for any $\tau\in(-1,1)$, $p,q>0$ such that $1/p+1/q=1$, and $y\geq 0$,
\[
y^2=y^{1-\tau}y^{1+\tau}\leq \frac{y^{p(1-\tau)}}{p} + \frac{y^{q(1+\tau)}}{q}.
\]
Choosing $p:=\frac{\alpha}{2(\alpha-1)}$, $q:=\frac{\alpha}{2-\alpha}$, we get 
$y^2\leq \frac{2(\alpha-1)}{\alpha}y^{\alpha}+\frac{2-\alpha}{\alpha}y^{2\alpha}$ which is further bounded above by $2c_\alpha y^\alpha+c_\alpha^2 y^{2\alpha}$ for $c_\alpha=\frac{\alpha-1}{\alpha}\vee \sqrt{\frac{2-\alpha}{\alpha}}$.
\end{proof}

\begin{lemma}
\label{lemma:lipschitz}
Functions $\psi_1(x)$ and $\psi_2(x)$ defined in Remark \ref{remark:psi2} are operator Lipschitz, with Lipschitz constants independent of the dimension. 
\end{lemma}
\begin{proof}
Lipshitz property of $\psi_1(x)$ follows from Theorem 1.6.1 in \cite{aleksandrov2016operator}. 
Result for $\psi_2(x)$ follows from Theorem 1.1.1 in the same paper. 
\end{proof}

\subsection{Proof of Lemma \ref{lemma:dilation}}
\label{proof:dilation}

For a self-adjoint matrices $R,Q$, $\|R\|\geq \|Q\|$ iff $\|R^2\|\geq \|Q^2\|$.  
Clearly, 
\[
\begin{pmatrix}
S & A \\
A^\ast & T
\end{pmatrix}^2 = 
\begin{pmatrix}
S^2 + AA^\ast &  SA + AT \\
A^\ast S + T A^\ast & T^2 + A^\ast A
\end{pmatrix}
\]
It implies that 
$\l\|\begin{pmatrix}
S & A \\
A^\ast & T
\end{pmatrix}^2 \r\| \geq 
\l\| S^2 +AA^\ast \r\|\geq \|AA^\ast\|$ 
and 
$\l\|\begin{pmatrix}
S & A \\
A^\ast & T
\end{pmatrix}^2 \r\| \geq 
\l\| T^2 +A^\ast A \r\|\geq \|A^\ast A\|$. 
Since 
$\begin{pmatrix}
0 & A \\
A^\ast & 0
\end{pmatrix}^2=
\begin{pmatrix}
AA^\ast & 0 \\
0 & A^\ast A
\end{pmatrix}
$, we obtain 
\[
\l\|\begin{pmatrix}
S & A \\
A^\ast & T
\end{pmatrix}^2 \r\|
\geq
\l\| \begin{pmatrix}
0 & A \\
A^\ast & 0
\end{pmatrix}^2\r\|,
\]
and result follows.
\qed

The following lemma is a generalization of Chebyshev's association inequality. 
\begin{lemma}[FKG inequality]
\label{lemma:FKG}
Let $f,g:\mathbb{R}^d\rightarrow\mathbb{R}$ be two functions that are non-decreasing with respect to each coordinate. 
Moreover, let $V=(V_1,V_2,\ldots,V_d)$ be a random vector taking values in $\mathbb{R}^d$. 
Then 
\[
\mb E f(V)g(V)\geq\mb E f(V) \mb E g(V).
\]
\end{lemma}
\begin{proof}
See Theorem 2.15 in \citep{boucheron2013concentration}.
\end{proof}
The following corollary is immediate.
\begin{corollary}
\label{corollary:FKG-bound}
Let $Z\in \mb R^d$ be a centered random vector with covariance matrix $\Sigma$. Then
\[
\sigma_0^2 := \l\| \mb E \|Z\|_2^2 Z Z^T \r\|\geq \mb E\|Z\|_2^2  \left\| \mb E Z Z^T \right\|
=\tr \Sigma \,\|\Sigma\|.
\]
\end{corollary}
\begin{proof}
Consider any unit vector $\mathbf{v}\in\mathbb{R}^d$. 
It is enough to show $\mb E\l[ (\mathbf{v}^TZ)^2\|Z\|_2^2 \r] \geq \mb E (\mathbf{v}^TZ)^2 \mb E \|Z\|_2^2$. 
We make the change the coordinates by considering an orthonormal basis 
$\{\mathbf{v}_1,\cdots,\mathbf{v}_d\}$ with 
$\mathbf{v}_1=\mathbf{v}$. 
Letting $V_i = \mathbf{v}_i^T Z$, $i=1,2,\cdots,d$, we obtain
\[
\mb E (\mathbf{v}^T Z)^2\|Z\|_2^2 = \mb E V_1^2\|V\|_2^2\geq \mb E V_1^2 \mb E \|V\|_2^2,
\]
where the last inequality follows from Lemma \ref{lemma:FKG} inequality by setting 
$f\left(V_1,\ldots,V_d\right) := V_1^2$ and $g\left(V_1,\ldots,V_d\right) := \|V\|_2^2$.
\end{proof}

\begin{lemma}
\label{lemma:lower}
Let $\gamma_1,\ldots,\gamma_n$ be independent random variables with density $p(x)=e^{-2|x|}, \ x\in \mb R$. 
Then for all $n>1$
\begin{enumerate}
\item $\mb E\max\l( |\gamma_1|,\ldots,|\gamma_n| \r)\geq \frac 1 2\log n$.
\item $\mb E\max\l( \gamma_1,\ldots,\gamma_n \r)\geq \frac 1 4\log n$.
\item $\Pr\l( \max\l( |\gamma_1|,\ldots,|\gamma_n| \r) \geq \l(\frac 1 2 - \tau \r) \log n \r)\geq c(\tau)>0$ for every $0<\tau<1/2$.
\end{enumerate}

\end{lemma}
\begin{proof}
Note that, since the distribution of $\gamma_j$'s is symmetric and $\mb E\max\l( \gamma_1,\ldots,\gamma_n \r)$ is positive, 
\begin{align*}
\mb E\max\l( |\gamma_1|,\ldots,|\gamma_n| \r) &
= \mb E\max\l( \max\l( \gamma_1,\ldots,\gamma_n \r), \max\l( -\gamma_1,\ldots,-\gamma_n \r)   \r) \\
&
\leq 2 \mb E\max\l( \gamma_1,\ldots,\gamma_n \r).
\end{align*}
Next, for $1\leq j\leq n$, $|\gamma_j|$ has exponential distribution with density $\tilde p(x)=2 e^{-2x}I\{ x\geq 0\}$. 
It follows from a well-known fact that 
$\mb E\max\l( |\gamma_1|,\ldots,|\gamma_n| \r) = \frac{1}{2}\sum_{j=1}^n \frac{1}{j}\geq \frac{1}{2}\log n$, and the first and second inequalities follow. 

A standard computation shows that $\mb E\max^2\l(|\gamma_1|,\ldots,|\gamma_n| \r)\leq c_1 \log^2 n$ for some numerical constant $c_1>0$, hence Paley-Zygmund inequality implies the last claim. 
\end{proof}

\section{Tools from probability theory and linear algebra}

We recall several useful results that we will need in the proofs below. 
		
\begin{lemma}[Matrix Hoeffding inequality]
\label{lemma:hoeffding}
Let $Z_1,\ldots,Z_n\in \mb C^{d\times d}$ be a sequence of independent self-adjoint random matrices such that for all $1\leq k\leq n$,
\[
\mb EZ_k = 0 \text{ and } \l\| Z_k \r\|\leq M_k \text{ almost surely}.
\]
Then 
$
\l\|  \sum_{j=1}^n Z_j \r\| \leq t
$
with probability $\geq 1-2d \exp\l( -\frac{t^2}{8\sum_{j=1}^n M_j^2} \r)$.
\end{lemma}
\begin{proof}
See Theorem 1.3 in \cite{tropp2012user}.
\end{proof}
We conclude this section by recalling the notion of Talagrand's generic chaining complexity (see \cite{talagrand2014upper}) and several related results. 
Given a metric space $(T,\rho)$, 
let $\left\{\Delta_n\right\}$ be a nested sequence of partitions of $T$ such that 
${\rm card}\,\Delta_0=1$ and ${\rm card}\,\Delta_n\leq 2^{2^n}$. 
For $s\in T$, let $\Delta_n(s)$ be the unique subset of $\Delta_n$ containing $s$. 
The generic chaining complexity $\gamma_2(T,\rho)$ is defined as 
\[
\gamma_2(T,\rho):=\inf_{\left\{\Delta_n\right\}}\sup_{s\in T}\sum_{n\geq 0} 2^{\frac n2}D(\Delta_n(s))
\]
where the infimum is taken over all admissible sequences of partitions and $D(A):=D(A,\rho)$ stands for the diameter of a set $A$.
The covering number $N(T,\rho,\eps)$ is defined as the smallest $N\in \mb N$ such that there exists a subset $F\subseteq T$ of cardinality $N$ with the property that for all $z\in T$, $\rho(z,F)\leq \eps$. 
Dudley's entropy integral bound (see \cite{talagrand2014upper}) states that
\begin{align}
\label{eq:dudley}
\gamma_2(T,\rho)\leq \frac{1}{2\sqrt 2-1}\int\limits_{0}^{D(T)} \sqrt{\log N(T,\rho,\eps/4)}d\eps.
\end{align}
We will say that $\mb C^{d\times d}$-valued stochastic process $\l\{ X(t), \ t\in T \r\}$ has sub-Gaussian increments with respect to the metric 
$\rho$ if for all $t_1,t_2\in \mb T$, 
\[
\Pr\l( \l\| X_{t_1}-X_{t_2} \r\|\geq s \rho(t_1,t_2) \r) \leq 2d e^{-s^2/2},
\]
where $\|\cdot\|$ is the operator norm. 
\begin{lemma}
\label{lemma:generic-chaining}
Let $(T,\rho)$ be a metric space and let $\mb C^{d\times d}$-valued stochastic process $\l\{ X(t), \ t\in T \r\}$ have sub-Gaussian increments with respect to $\rho$. 
There exists an absolute constant $C>0$ such that for any $t_0\in T$ and any $s\geq 1$, 
\[
\sup_{t\in T} \| X_t - X_{t_0}\|\leq C\l( \gamma_2(T,\rho) + \sqrt{s} D(T) \r)
\]
with probability $\geq 1 - 2d e^{-s}$. 
\end{lemma}
\begin{proof}
See Theorem 3.2 in \cite{dirksen2013tail} for a more general statement.
\end{proof}


\section{Proof of Theorem \ref{th:intdim}}
\label{proof:intdim}

Define $\phi(x)=\max(e^{x}-1,0)$ and $X_j=\psi(\theta Y_j)$. 
Proceeding as in the proof of Theorem \ref{th:main}, we get that for $s\geq 0$,
\begin{align*}
\Pr\Bigg( \lambda_{\mx}&\l( \frac{1}{\theta}\sum_{j=1}^n \l(X_j - \theta\mb E Y_j\r) \r) \geq s \Bigg)  =
\Pr\l( \phi\l( \lambda_{\mx}\l( \sum_{j=1}^n \l(X_j - \theta\mb E Y_j\r) \r)  \r) \geq \phi(\theta s) \r) \\
&
\leq \frac{1}{\phi(\theta s)}\mb E \tr \phi\l( \sum_{j=1}^n \l( X_j -\theta\mb EY_j\r) \r) 
= \frac{1}{\phi(\theta s)} \l(  \mb E \tr \exp\l( \sum_{j=1}^n \l( X_j -\theta\mb EY_j\r) \r) - I\r).
\end{align*}
It follows from Lemma \ref{lemma:main} that 
\[
\mb E \tr \exp\l( \sum_{j=1}^n \l( X_j -\theta\mb EY_j\r) \r)  \leq \tr\exp\l(  \frac{\theta^2}{2}\sum_{j=1}^n \mb EY_j^2  \r). 
\]
Set $B_n^2:=\sum_{j=1}^n \mb EY_j^2\succeq 0$, and note that 
\begin{align*}
&
\tr\l[\exp\l(  \frac{\theta^2}{2}\sum_{j=1}^n \mb EY_j^2  \r) - I \r]=
\tr\l[  \frac{\theta^2}{2}\sqrt{B_n^2}\l( I + \frac{\frac{\theta^2}{2}B_n^2}{2!}+\ldots+  \frac{\l(  \frac{\theta^2}{2} B_n^2\r)^{k-1}}{k!} +\ldots \r)\sqrt{B_n^2} \r] \\
&
\leq \tr \l[ \frac{\theta^2}{2}B_n^2\l( 1 + \frac{\frac{\theta^2}{2}\|B_n^2\|}{2!}+\ldots+  \frac{\l(  \frac{\theta^2}{2} \|B_n^2\|\r)^{k-1}}{k!} +\ldots\r)  \r]
=\frac{\tr B_n^2}{\| B_n^2 \|} \l(\exp\l( \frac{\theta^2}{2}\|B_n^2\|  \r) -1 \r).
\end{align*}
Here we have used the fact that $A\preceq B$ implies $SAS^\ast\preceq SBS^\ast$ for $S=S^\ast:=\sqrt{B_n^2}$, and the equality 
$\frac{e^x-1}{x}=\sum_{j=1}^\infty \frac{ x^{j-1} }{j!}$. 
We have shown that 
\begin{align*}
&
\Pr\Bigg( \lambda_{\mx}\l( \frac{1}{\theta}\sum_{j=1}^n \l(X_j - \theta\mb E Y_j\r) \r) \geq s \Bigg) \leq
\frac{\tr B_n^2}{\| B_n^2 \|} \frac{ \exp\l( \frac{\theta^2}{2}\|B_n^2\|  \r) -1 }{e^{\theta s}-1} \\
&
\leq  \frac{\tr B_n^2}{\| B_n^2 \|} \exp\l( \frac{\theta^2}{2}\|B_n^2\| - \theta s \r)  \frac{e^{\theta s}}{e^{\theta s}-1} 
\leq  \frac{\tr B_n^2}{\| B_n^2 \|} \exp\l( \frac{\theta^2}{2}\|B_n^2\| - \theta s \r) \l( 1 + \frac{1}{\theta s} \r), 
\end{align*}
where we used an elementary inequality $\frac{e^{\theta s}}{e^{\theta s}-1}\leq 1 +\frac{1}{\theta s}$ on the last step. 

Combining the same steps with Fact \ref{fact:04} and the equality $-\lambda_{\mn}(A)=\lambda_{\mx}(-A)$, we get
\begin{align*}
\Pr\Bigg( \lambda_{\mn}\l( \frac{1}{\theta}\sum_{j=1}^n \l(X_j - \theta\mb E Y_j\r) \r) \leq  -s \Bigg) \leq 
\frac{\tr B_n^2}{\| B_n^2 \|} \exp\l( \frac{\theta^2}{2}\|B_n^2\| - \theta s \r) \l( 1 + \frac{1}{\theta s} \r).
\end{align*}
Finally, replace $s$ by $t\sqrt n$ to get the bound in the required form.

\section{Proof of Theorem \ref{th:adaptive}}
\label{section:proof-adaptive}

Let $\m E_1$ be the event defined by 
\[
\m E_1=\l\{ \l\|  \hat T_0 - \mb EY \r\|\leq 6\sigma\sqrt{\frac{2t}{n/2}}  \r\}.
\]
By Theorem \ref{th:lepski},
\begin{align}
\label{eq:c05}
&
\Pr(\m E_1)\geq 1- 2d\log_2\l( \frac{2\sigma_{\mx}}{\sigma_{\mn}}\r)e^{-t}.
\end{align} 
Note that on this event,  
\begin{align}
\label{eq:c00}
\l\|\mb E \l[ (Y - \hat T_0)^2|\hat T_0 \r] \r\| & = \| \mb E(Y-\mb EY)^2 + (\mb EY - \hat T_0)^2 \| \leq 
\sigma_0^2 + \l(12 \sigma\sqrt{\frac t n}\r)^2.
\end{align}
In particular, on event $\m E_1$, 
\begin{align}
\label{eq:c10}
\l\|\mb E \l[ (Y - \hat T_0)^2|\hat T_0 \r] \r\|^{1/2}\leq \sigma_{0}+12 \sigma\sqrt{\frac t n} \leq \sigma_{0,\mx}+12 \sigma_{\mx}\sqrt{\frac t n}.
\end{align} 
Next,
\begin{align}
\label{eq:c20}
&\Pr\l(\l\| \hat T_1 -\mb EY \r\| \geq 6\l(\sigma_0 + 12\sigma\sqrt{\frac t n} \r)\sqrt{\frac{2t}{n/2}}  \r) \\
&\nonumber=
\Pr\l( \l\| T_{|G_2|,j_2^\ast}(\hat T_0;G_2) - (\mb EY - \hat T_0) \r\| \geq 6\l(\sigma_0 + 12\sigma\sqrt{\frac t n} \r)\sqrt{\frac{2t}{n/2}}  \r) \\
&\nonumber 
\leq
\Pr(\m E_1^c) + \Pr\l( \l\| T_{|G_2|,j_2^\ast}(\hat T_0;G_2) - (\mb EY - \hat T_0) \r\| \geq  
6\l(\sigma_0 + 12\sigma\sqrt{\frac t n} \r)\sqrt{\frac{2t}{n/2}}   \Big|\m E_1 \r).
\end{align}
Define the new probability measure by $\widetilde\Pr(A)=\Pr(A|\m E_1)$. 
Clearly, under this new measure, subsample $G_2$ is still independent of $G_1$ since $\m E_1\in \sigma(G_1)$ - the sigma-algebra generated by $G_1$, and for any $B\in \sigma(G_2)$, $\widetilde \Pr(B) = \Pr(B)$. 
Let $\tilde{\mb E}$ be the expectation with respect to measure $\widetilde \Pr(\cdot)$. 
Then
\begin{align*}
\Pr\bigg(& \| T_{|G_2|,j_2^\ast}(\hat T_0;G_2) - (\mb EY - \hat T_0) \|\geq  6\l(\sigma_0 + 12\sigma\sqrt{\frac t n} \r)\sqrt{\frac{2t}{n/2}}  \Big|\m E_1 \bigg) 
\\
&= 
\widetilde\Pr\l( \| T_{|G_2|,j_2^\ast}(\hat T_0;G_2) - (\mb EY - \hat T_0) \|\geq  6\l(\sigma_0 + 12\sigma\sqrt{\frac t n} \r)\sqrt{\frac{2t}{n/2}}  \r) \\
&\leq 
\widetilde\Pr\l( \| T_{|G_2|,j_2^\ast}(\hat T_0;G_2) - (\mb EY - \hat T_0) \|\geq  6\l\|\mb E \l[ (Y - \hat T_0)^2|\hat T_0 \r] \r\|^{1/2} \sqrt{\frac{2t}{n/2}}  \r)\\
&
=\tilde{\mb E} \widetilde \Pr\l( \l\| T_{|G_2|,j_2^\ast}(\hat T_0;G_2) - (\mb EY - \hat T_0) \r\|\geq 
6\l\|\mb E \l[ (Y - \hat T_0)^2|\hat T_0 \r] \r\|^{1/2} \sqrt{\frac{2t}{n/2}}  \Big| \hat T_0 \r) \\ 
&\leq 
1 - 2d\log_2\l(\frac{2(\sigma_{0,\mx}+12\sigma_{\mx}\sqrt{t/n})}{\sigma_{0,\mn}} \r) e^{-t}. 
\end{align*}
Here, we use the definition of $\widetilde \Pr(\cdot)$ on the first step and (\ref{eq:c10}) on the second step. 
The last inequality follows from independence of $G_2$ from $\hat T_0$ (under $\widetilde \Pr(\cdot)$) and Theorem \ref{th:lepski} applied conditionally on $\hat T_0$: indeed, this can be done since (\ref{eq:c10}) holds on $\m E_1$. 
It remains to combine the last bound with (\ref{eq:c20}) and (\ref{eq:c05}). 
\qed

\section{Proof of Theorem \ref{th:iterative}}
\label{proof:iterative}
 
We will first state several technical results that are required in the proof. 
Let $j\in \m J$ be such that $\sigma_{0,j}=2^j\sigma_{0,\mn}\geq \sigma_0$, and define 
\[
X_{j,i}(S):=\psi\l(\theta_j (Y_i - S)  \r), \ i=1,\ldots,n.
\]
Moreover, set
\begin{align*}
&
L_n(\delta,j):= 
\sup\limits_{S:\|S-\mb EY\|\leq \delta}
\l\|  \frac{1}{n\theta_j} \sum_{i=1}^n \l( X_{j,i}(S) - \mb E X_{j,i}(S) \r) - \frac{1}{n\theta_j}\sum_{i=1}^n \l( X_{j,i}(\mb EY) - \mb E X_{j,i}(\mb EY) \r)  \r\|,
\end{align*}
and define the event
\[
\Omega(\delta,j) = 
\l\{    L_n(\delta,j) \leq  K\delta\sqrt{\frac{d^2+L t}{n}}  \r\},
\]
where $K>0$ is an absolute constant. 
\begin{lemma}
\label{lemma:L_n}
For $K$ large enough, 
\[
\Pr \l( \Omega(\delta,j) \r)\geq 1-2d e^{-t}.
\]
\end{lemma}
\begin{proof}
See section \ref{proof:L_n}.
\end{proof}
\begin{lemma}
\label{lemma:norm}
For any Hermitian $S$,
\[
\l\|  S- \mb EY + \frac{1}{\theta_j}\mb E X_{j,1}(S)  \r\|\leq  \frac{\theta_j}{2}\l\| \mb E(Y - S)^2 \r\|.
\]
\end{lemma}
\begin{proof}
Note that
\begin{align*}
X_{j,1}(S)&=\psi\l( \theta_j(Y_1-S) \r)\preceq \log\l( I+\theta_j (Y_1 - S) +\frac{\theta_j^2}{2}(Y_1 - S)^2 \r)  \\
&\preceq 
\theta_j (Y_1 - S) +\frac{\theta_j^2}{2}(Y_1 - S)^2, 
\end{align*}
which is a consequence of scalar inequality $\log(1+x)\leq x, \ x>-1$ and fact M.2.1, 
hence we can deduce from fact M.2.2 that
\begin{align*}
&
\lambda_{\mx}\l(  S- \mb EY + \frac{1}{\theta_j}\mb E X_{j,1}(S)  \r)\leq \frac{\theta_j}{2}\l\| \mb E(Y - S)^2 \r\|.
\end{align*}
At the same time, 
\[
-\lambda_{\mn}\l(  S- \mb EY + \frac{1}{\theta_j}\mb E X_{j,1}(S)  \r) = 
\lambda_{\mx}\l( -\frac{1}{\theta_j}\mb E X_{j,1}(S) - (S - \mb EY)  \r).
\] 
Since 
\begin{align*}
-X_{j,1}(S) & = -\psi\l( \theta_j(Y_1-S) \r)\preceq
\log\l( I-\theta_j (Y_1 - S) +\frac{\theta_j^2}{2}(Y_1-S)^2 \r)
\end{align*}
by the definition of $\psi(\cdot)$, we conclude that 
\[
-\lambda_{\mn}\l(  S- \mb EY - \frac{1}{\theta_j}\mb E X_{j,1}(S)  \r) \leq \frac{\theta_j}{2}\l\| \mb E(Y - S)^2 \r\|,
\]
hence 
$\l\|  S- \mb EY + \frac{1}{\theta_j}\mb E X_{j,1}(S)  \r\|\leq  \frac{\theta_j}{2}\l\| \mb E(Y - S)^2 \r\|$.
\end{proof}
\begin{lemma}
\label{lemma:expect}
With probability $\geq 1-2de^{-t}$,
\begin{align*}
\l\| \frac{1}{n\theta_j}\sum_{i=1}^n \l( X_{j,i}(\mb EY) - \mb E \big[X_{j,i}(\mb EY)\big] \r)  \r\| 
\leq \sigma_{0,j}\sqrt{\frac{2t}{n}} + \frac{\theta_j}{2}\sigma_0^2.
\end{align*}
\end{lemma}
\begin{proof}
Result follows from Theorem M.3.1 and the inequality
\[
\l\| \frac{1}{\theta_j} \mb E X_{j,1}(\mb EY) \r\|\leq \frac{\theta_j}{2}\sigma_0^2,
\] 
which is a consequence of lemma \ref{lemma:norm}. 
Indeed,
\begin{align*}
\l\| \frac{1}{n\theta_j}\sum_{i=1}^n \l( X_{j,i}(\mb EY) - \mb E \big[ X_{j,i}(\mb EY) \big] \r)  \r\| &\leq 
\l\| \frac{1}{n\theta_j}\sum_{i=1}^n X_{j,i}(\mb EY) \r\| + 
\frac{1}{\theta_j}\l\|  \mb E X_{j,1}(\mb EY) \r\| \\
&
\leq \sigma_{0,j}\sqrt{\frac{2t}{n}} + \frac{\theta_j}{2}\sigma_0^2
\end{align*}
with probability $\geq 1-2de^{-t}$. 
\end{proof}
\noindent 
We are ready to proceed with the proof of the theorem. 
Let 
\[
\m E_0=\l\{ \l\| T_{n}^{(0)} - \mb EY \r\|\leq \sigma_{\mx}\sqrt{\frac{2t}{n}} \r\},
\] 
where $T_n^{(0)}$ was defined in M.6.2 as $T_n^{(0)} = \frac{1}{n\theta}\sum_{i=1}^n \psi\l( \theta Y_i \r)$, 
and note that $\Pr(\m E_0)\geq 1-2d e^{-t}$ by Theorem M.3.1.
Let 
\begin{align}
\label{eq:k_max}
k_{\max}&=1+\max\l\{ k\geq 0: \ \sigma_{0,\mn} 1.1^k \leq \frac{12}{5} \sigma_{\mx} \r\}, \\
\nonumber
\gamma_l &= 1.1^{l}\sigma_{0,\mn}\sqrt{\frac{2t}{n}}, \  l\geq 0,
\end{align} 
and note that $k_{\max}\leq 1+\l\lfloor  \frac{\log_2 \l(  \frac{12\sigma_{\mx}}{5\sigma_{0,\mn}}\r)}{\log_2 1.1} \r\rfloor \leq 1+8\log_2 \l(  \frac{12\sigma_{\mx}}{5\sigma_{0,\mn}}\r)$.
Define
\[
\Omega_j:= \l\{  \l\| \frac{1}{n\theta_j}\sum_{i=1}^n \l( X_{i,\mb EY} - \mb E X_{i,\mb EY} \r)  \r\| 
\leq \sigma_{0,j}\sqrt{\frac{2t}{n}} + \frac{\theta_j}{2}\sigma_0^2 \r\} \cap
\bigcap_{l=0}^{k_{\max}} \Omega\l(\gamma_l,j\r).
\]
By Lemma \ref{lemma:L_n}, Lemma \ref{lemma:expect} and the union bound, $\Pr\l( \Omega_j \r) \geq 1 - 2d(2+k_{\max})e^{-t}$. 
We will now show by induction that on the event $\m E_0 \cap \Omega_j$, 
$\l\| T_{n,j}^{(k)} - \mb EY \r\|\leq \delta_j^{(k)}$ for all $k\geq 0$. 
For $k=0$, result follows from the definition of $\m E_0$. 
In remains to complete the induction step $k-1\mapsto k$. 
Note that when $\l\{ \l\| T_{n,j}^{(k-1)} - \mb EY \r\|\leq \delta_j^{(k-1)} \r\}$ occurs, we have
\begin{align}
\nonumber
&
\l\| T^{(k)}_{n,j} - \mb EY \r\| = 
\l\|  T^{(k-1)}_{n,j} - \mb EY +\frac{1}{n\theta_j} \sum_{i=1}^n \psi\l(\theta_j (Y_i-  T^{(k-1)}_{n,j}) \r) \r\| \\
&
\label{eq:00}
\leq \sup_{S: \|S-\mb EY\|\leq \delta^{(k-1)}_{j}} 
\l\|  S - \mb EY + \frac{1}{n\theta_j} \sum_{i=1}^n \psi\l( \theta_j (Y_i - S) \r) \r\|.
\end{align}
Expression under the supremum in \eqref{eq:00} can be decomposed as follows: 
\begin{align*}
&
S - \mb EY + \frac{1}{n\theta_j} \sum_{i=1}^n \psi\l( \theta_j (Y_i - S) \r) = \\
&
S- \mb EY + \frac{1}{\theta_j}  \mb E X_{j,1}(S) 
+\frac{1}{n\theta_j} \sum_{i=1}^n \l( X_{j,i}(S) - \mb E X_{j,i}(S) \r)  \\ 
& 
- \frac{1}{n\theta_j}\sum_{i=1}^n \l( X_{j,i}(\mb EY) - \mb E X_{j,i}(\mb EY) \r)  
+ \frac{1}{n\theta_j}\sum_{i=1}^n \l( X_{j,i}(\mb EY) - \mb E X_{j,i}(\mb EY) \r).
\end{align*}
We will treat 3 terms separately: first, it follows from Lemma \ref{lemma:norm} that on $\Omega_j$
\begin{align}
\label{eq:b00}
\sup_{S:\|S-\mb EY\|\leq \delta^{(k-1)}_{j}} \l\|  S- \mb EY + \frac{1}{\theta_j}\mb E X_{1,S}  \r\| \leq \frac{\theta_j}{2}\l( \sigma_0^2+\l(\delta^{(k-1)}_j\r)^2\r).
\end{align}
Next, 
\begin{align}
\label{eq:b10}
\l\| \frac{1}{n\theta_j}\sum_{i=1}^n \l( X_{i,\mb EY} - \mb E X_{i,\mb EY} \r)  \r\| 
\leq \sigma_{0,j}\sqrt{\frac{2t}{n}} + \frac{\theta_j}{2}\sigma_0^2,
\end{align}
once again by the definition of $\Omega_j$. 
Let $\tilde l=\min\l\{l\geq 0: \ \gamma_l \geq \delta^{(k-1)}_j \r\}$ (where $\gamma_l$ was defined in \eqref{eq:k_max}), and note that $\tilde l\leq k_{\max}$ and $\gamma_{\tilde l}\leq 1.1\delta^{(k-1)}_j$.
We bound the third term as
\begin{align}
\label{eq:b20}
&
\sup_{S:\|S-\mb EY\|\leq\delta^{(k-1)}_j} 
\l\| \frac{1}{n\theta_j} \sum_{i=1}^n \l( X_{j,i}(S) - \mb E X_{j,i}(S) \r) - \frac{1}{n\theta_j}\sum_{i=1}^n \l( X_{j,i}(\mb EY) - \mb E X_{j,i}(\mb EY) \r) \r\| \\ & \nonumber
= L_n\l( \delta^{(k-1)}_j,j \r)
\leq  L_n\l( \gamma_{\tilde l},j \r) \leq K\gamma_{\tilde l}\sqrt{\frac{d^2+L t}{n}}\leq 1.1 K  \delta^{(k-1)}_j \sqrt{\frac{d^2+L t}{n}}.
\end{align}
Putting the bounds \eqref{eq:b00},\eqref{eq:b10},\eqref{eq:b20} together, we can estimate the supremum in (\ref{eq:00}) as
\begin{align}
\nonumber
&
\sup_{S: \|S-\mb EY\|\leq  \delta^{(k-1)}_j} 
\l\|  S - \mb EY +\frac{1}{n\theta_j} \sum_{i=1}^n \psi\l( \theta_j (Y_i - S) \r) \r\| \\
& \nonumber 
\leq 
\frac{\theta_j}{2}\l(\sigma_0^2+\l(  \delta^{(k-1)}_j \r)^2 \r) + \frac{\theta_j}{2}\sigma_0^2 + 
\sigma_{0,j} \sqrt{\frac{2t}{n}} +  \delta^{(k-1)}_j\cdot 1.1 K\sqrt{\frac{d^2+L t}{n}} \\
&
\label{eq:b30} 
\leq \l(\sigma_0+\sigma_{0,j}\r)\sqrt{\frac{2t}{n}}+ \delta^{(k-1)}_j\l(1.1 K \sqrt{\frac{d^2+L t}{n}} + \sqrt{\frac{2t}{n}}\frac{1}{2\sigma_{0,j}} \r).
\end{align}
Note that we have used bounds $\theta_j \sigma_0^2\leq \sigma_0\sqrt{\frac{2t}{n}}$ and $\theta_j\l(  \delta^{(k-1)}_j \r)^2\leq \theta_j \delta^{(k-1)}_j$ (indeed, inequality \eqref{eq:e_k} implies that $\delta_j^{(m)}\leq 1$ for all $j$ and $m$) to get the second inequality above. 
Since $j$ was chosen such that $\sigma_{0,j}\geq \sigma_0$ and $\tau = 1.1 K \sqrt{\frac{d^2+L t}{n}} + \sqrt{\frac{2t}{n}}\frac{1}{2\sigma_0}\leq \frac{1}{6}$ by assumption, we have shown that 
\[
\l\| T^{(k)}_{n,j} - \mb EY \r\| \leq 2\sigma_{0,j}\sqrt{\frac{2t}{n}} + \frac{1}{6}\delta_j^{(k-1)}=\delta_j^{(k)},
\]
where the last equality follows from the fact that the sequence $\delta_j^{(k)}$ defined in \eqref{eq:delta_k} satisfies the recursive relation 
\[
\delta_j^{(0)} = \sigma_{\mx}\sqrt{\frac{2t}{n}}, \quad 
\delta_j^{(k)} = 2\sigma_{0,j}\sqrt{\frac{2t}{n}} + \frac{1}{6}\delta_j^{(k-1)}.
\]
To complete the proof, it is enough to follow the steps of the proof Theorem \ref{th:lepski} applied to the collection of estimators 
$\l\{ T_{n,j}^{(k)}: \ j\in \m J\r\}$: 
first, let $\bar j=\min\l\{ j\in \m J: \sigma_{0,j}\geq \sigma_0\r\}$, and note that the event 
\[
\m E_0\cap\bigcap_{j\geq \bar j, \ j\in \m J}\Omega_j
\]
has probability $\geq 1 - 8d\l( 1 + 2\log_2\l( \frac{12\sigma_{\mx}}{5\sigma_{0,\mn}}\r)\r)\log_2\l( \frac{2\sigma_{0,\mx}}{\sigma_{0,\mn}} \r) e^{-t}$. 
Moreover, on this event $j_k^\ast\leq \bar j$, hence 
\begin{align*}
\l\| \hat T_k - \mb EY \r\| & 
= \l\| T_{n,j_k^\ast}^{(k)} -\mb EY \r\|  
\leq \l\| T_{n,j_k^\ast}^{(k)} - T_{n,\bar j}^{(k)} \r\| + \l\| T_{n,\bar j}^{(k)} -\mb EY\r\| \\
&
\leq 3\delta_{\bar j}^{(k)}\leq 3\l[ (1-6^{-k})\frac{24}{5}\sigma_{0}\sqrt{\frac{2t}{n}} + 6^{-k}\sigma_{\mx}\sqrt{\frac{2t}{n}} \r],
\end{align*}
where we used the fact that $\sigma_{0,\bar j}\leq 2\sigma_0$ in the last inequality.
\qed

\subsection{Proof of Lemma \ref{lemma:L_n}}
\label{proof:L_n}
 
To this end, we will use a chaining argument. 
Recall that the function $\psi(\cdot)$ is operator Lipschitz with Lipschitz constant $L$ by assumption. 
Recall that $X_{j,i}(S):=\psi\l(\theta_j (Y_i - S)  \r), \ i=1,\ldots,n$.
It follows from Assumption \ref{ass:2} (see also Lemma \ref{lemma:lipschitz}) that for any Hermitian $S_1,S_2$ and $1\leq i\leq n$,
\begin{align*}
\| X_{i,S_1} - X_{i,S_2} \| &
= \|  \psi(\theta_j(Y_i-S_1)) - \psi(\theta_j(Y_i - S_2)  \| \leq 
L\theta_j \|S_1 - S_2 \|.
\end{align*}
Matrix Hoeffding's inequality (Lemma \ref{lemma:hoeffding}) applies with 
\[
Z_i = \frac{1}{n\theta_j}\l( (X_{i,S_1}- \mb E X_{i,S_1}) - (X_{i,S_2}- \mb E X_{i,S_2})  \r), \ i=1,\ldots,n,
\] 
and
$M_i = \frac{2 L}{n}\| S_1 - S_2 \|$, and yields that
\[
\l\|  \sum_{i=1}^{n} Z_i  \r\| \leq L\sqrt{32}\| S_1 - S_2 \| \sqrt{\frac{s}{n}}
\]
with probability $\geq 1-2d e^{-s}$. 
\begin{lemma}[Covering number in the operator norm]
\label{lemma:cover}
Let $B(r)$ be the ball of radius $r>0$ in $\mb R^{d^2}$ with respect to the operator norm $\|\cdot\|$, centered at $0$. 
Then the covering number $N(B(r),\eps):=N(B(r),\|\cdot\|,\eps)$ satisfies
\[
N(B(r),\eps)\leq \l(\frac{2r}{\eps}+1\r)^{d^2}.
\]
\end{lemma}
\begin{proof}
It is well known \cite{VershyninNon-Asymptotic-00} that 
\[
N( A,\eps)\leq  \frac{\l| A + B(\eps/2) \r|}{\l| B(\eps/2) \r|},
\] 
where $|C|$ denotes the Lebesgue measure of a set $C$, and $A+C$ stands for the Minkowski sum of the sets $A$ and $C$. 
For $A=B(r)$, we get $N(B(r),\eps)\leq \frac{\l| B(r+\eps/2) \r|}{\l| B(\eps/2) \r|}$. 
The volume of the unit ball is given by 
\[
|B(r)|=c_d \int\limits_{[-r,r]^d} \prod\limits_{1\leq i<j\leq d}\l| x_i^2 - x_j^2 \r| dx_1\ldots dx_d,
\]
where $c_d=d! 4^{-d}\l(\prod_{j=1}^d v_j^2\r)^2$ and $v_j$ is the volume of the Euclidean unit ball in $\mb R^j$. 
From here, it is easy to see that 
\[
\frac{\l| B(r+\eps/2) \r|}{\l| B(\eps/2) \r|}= \l(  \frac{2r}{\eps} +1 \r)^{d^2}.
\]
\end{proof}

\noindent Let $T(\delta_{k-1}):=\l\{ S\in \mb C^{d\times d}: \ \|S-\mb EY\|\leq \delta_{k-1}  \r\}$, and define the metric 
\[
\rho_d(S_1,S_2):=L\| S_1-S_2 \|, \ S_1,S_2\in \mb C^{d\times d}.
\] 
Viewing $S\mapsto \frac{1}{n\theta_j}\sum_{i=1}^n (X_{i,S}- \mb E X_{i,S})$ as a $\mb C^{d\times d}$-valued stochastic process indexed by the elements of the metric space $(T(\delta_{k-1}),\rho_d)$, we can apply Lemma \ref{lemma:generic-chaining} which implies that there exists an absolute constant $C>0$ such that for any $t\geq 1$,	
\begin{align}
\label{eq:chain}
L_n(\delta_{k-1})\leq \frac{C}{\sqrt n}\l(  \gamma_2(T(\delta_{k-1}),\rho_d) + \sqrt t D(T(\delta_{k-1}),\rho_d) \r)
\end{align}
with probability $\geq 1-2de^{-t}$.
Recall the Dudley's entropy integral bound (\ref{eq:dudley}):
\[
\gamma_2(T(\delta_{k-1}),\rho_d)\leq \frac{1}{2\sqrt 2-1}\int\limits_{0}^{D(T(\delta_{k-1}),\rho_d)} \sqrt{\log N(T(\delta_{k-1}),\rho_d,\eps/4)}d\eps.
\]
Noting that $D(T(\delta_{k-1}),\rho_d)=2L \, \delta_{k-1}$ and combining Dudley's bound with the estimate of Lemma \ref{lemma:cover}, we get 
\[
\gamma_2(T(\delta_{k-1}),\rho_d)\leq C_1 \, L \,\delta_{k-1} d, 
\]
where $C_1=\frac{2}{2\sqrt 2-1}\int_0^1 \log^{1/2}(1+4/\eps)d\eps$. 
Bound (\ref{eq:chain}) implies that with probability $\geq 1-2de^{-t}$, 
\begin{align}
\label{eq:b120}
&
L_n(\delta_{k-1})\leq \frac{C}{\sqrt n}\l(   C_1 L \delta_{k-1} d + 2 L \delta_{k-1}\sqrt{t}  \r)
\leq \delta_{k-1} \cdot K\sqrt{\frac{d^2+L t}{n}}
\end{align}
for some absolute constant $K>0$.

\end{document}